\documentclass[11pt, oneside]{article}  %
\usepackage[utf8]{inputenc}
\usepackage{booktabs}
\usepackage{geometry}          %
\usepackage{hyperref}
\usepackage[perpage]{footmisc}
\usepackage{graphicx}		   %
\usepackage{amssymb}
\usepackage{amsfonts}
\usepackage{amsthm}
\usepackage{algorithm}
\usepackage{algpseudocode}
\usepackage{amsmath}
\usepackage{amscd}
\usepackage{amsbsy}            %
\usepackage{psfrag}            %
\usepackage{epsf}              %
\usepackage{graphicx}          %
\usepackage{makeidx}           %
\usepackage{color}             %
\usepackage{fancyhdr}
\newtheorem{thm}{Theorem}[section]
\newtheorem{exm}[thm]{Example}

\newtheorem{lem}[thm]{Lemma}
\newtheorem{definition}[thm]{Definition}

\newtheorem{prop}[thm]{Proposition}
\newtheorem{cor}[thm]{Corollary}

\usepackage[margin=10pt,font=small,labelfont=bf]{caption} 
\usepackage[utf8]{inputenc} 
\usepackage{xcolor} 
\usepackage{listings}

\definecolor{mygreen}{RGB}{28,172,0}
\definecolor{mylilas}{RGB}{170,55,241}

\newcommand{\abs}[1]{\left\vert#1\right\vert}

\newcommand{\dsb}[1]{[\![#1]\!]}
\newcommand{\norm}[1]{\parallel\! #1\! \parallel}
\newcommand{\partl}{\partial}
\newcommand{\rank}{\texttt{rank}}

\newcommand{\seq}[1]{\left<#1\right>}
\newcommand{\set}[1]{\left\{#1\right\}}

\newcommand{\Sym}{\texttt{sym}}
\newcommand{\tr}{\texttt{tr}}
\newcommand{\vecc}{\texttt{vec}}

\newcommand{\al}{\alpha}

\newcommand{\be}{\beta}
\newcommand{\bld}[1]{\mathbf{#1}}

\newcommand{\la}{\lambda}
\newcommand{\La}{\Lambda}
\newcommand{\p}{\prime}

\newcommand{\si}{\sigma}

\newcommand{\varp}{\varepsilon}

\newcommand{\bbC}{\mathbb{C}}

\newcommand{\bbI}{\mathbb{I}}

\newcommand{\RR}{\mathbb{R}}

\newcommand{\TT}{\mathbb{T}}

\newcommand{\bbW}{\mathbb{W}}

\newcommand{\A}{\mathcal{A}}
\newcommand{\tA}{\tilde{\A}}

\newcommand{\B}{\mathcal{B}}
\newcommand{\C}{\mathcal{C}}
\newcommand{\caD}{\mathcal{D}}

\newcommand{\caF}{\mathcal{F}}
\newcommand{\caG}{\mathcal{G}}
\newcommand{\tG}{\tilde{\mathcal{G}}}

\newcommand{\caI}{\mathcal{I}}
\newcommand{\caJ}{\mathcal{J}}
\newcommand{\caK}{\mathcal{K}}

\newcommand{\caM}{\mathbf{M}}

\newcommand{\T}{\mathcal{T}}
\newcommand{\U}{\mathcal{U}}
\newcommand{\V}{\mathcal{V}}

\newcommand{\X}{\mathcal{X}}
\newcommand{\ttX}{\tilde{\X}}
\newcommand{\Y}{\mathcal{Y}}

\newcommand{\bfm}{\textbf{m}}
\newcommand{\bfn}{\textbf{n}}
\newcommand{\bfr}{\textbf{r}}
\newcommand{\bft}{\textbf{t}}

\newcommand{\bx}{\textbf{x}}

\newcommand{\by}{\textbf{y}}

\newcommand{\bfc}{\textbf{c}}

\newcommand{\rmI}{\rm{I}}

\newcommand{\mnts}{$m$th order $n$-dimensional tensors }
\newcommand{\mnrt}{$m$th order $n$-dimensional real tensor }

\newcommand{\mnst}{$m$th order $n$-dimensional symmetric tensor }

\newcommand{\beq}{\begin{equation}}
\newcommand{\eeq}{\end{equation}}
\newcommand{\bey}{\begin{eqnarray}}
\newcommand{\eey}{\end{eqnarray}}
\newcommand{\beyy}{\begin{eqnarray*}}
\newcommand{\eeyy}{\end{eqnarray*}}

\makeatletter

\@addtoreset{algorithm}{section}
\makeatother

\title{Tensor Derivatives, Unified Tensor-Form Differential Equations, and Model Reduction via Partial Tucker Decomposition}
\author{
Yiran Xu\thanks{Department of Mathematics and Statistics, Georgia State University, Atlanta, GA, USA. 
Email: yxu40@gsu.edu}
\and
Changqing Xu\thanks{Corresponding author. School of Mathematical Science, Suzhou University of Science and 
Technology, Suzhou, China. Email: cqxurichard@usts.edu.cn}}

\makeatletter
      \def\@setcopyright{}
      \def\serieslogo@{}
      \makeatother
\date{}
\begin{document}
\maketitle

\begin{abstract}
This paper develops a unified tensor calculus for matrix-valued functions and their derivatives, and leverages this framework to construct efficient 
model reduction techniques for high-dimensional tensor differential equations.  We first establish a systematic theory of tensor differentiation, 
wherein the derivative of a matrix with respect to another matrix is represented as a fourth-order tensor. Building on this calculus, we recast linear 
ordinary differential equations (ODEs) and partial differential equations(PDEs) into a compact tensor-matrix form $\frac{dX}{dt} = \A\ast X$. The 
general solution is expressed as $X = \exp(t\A)\ast C$, extending the matrix exponential to the tensor setting. Conditions under which the solution 
admits this exponential form are characterized in terms of the commutativity of the associated matrix slices. We introduce the partial Tucker 
decomposition (parTuckerD) to address the computational challenges posed by high-order tensor systems. On a synthetic electronic health record 
(EHR) tensor, parTuckerD achieves a relative reconstruction error of $0.0992$ with a $136.3\times$ compression ratio, matching the accuracy of the 
full TuckerD while preserving patient-level similarity structure. The results demonstrate that the proposed tensor calculus and parTuckerD framework 
provide a principle and computationally efficient approach for analyzing and solving high-dimensional tensor differential equations arising in data-intensive applications.
\end{abstract}

\noindent \textbf{keywords:} \  Contractive product; derivative; linear ordinary differential equation; outer product; Lyapunov function.


\section{Introduction}
\setcounter{equation}{0}

The theory of differential equations (DEs) was originated earlier in the 17th century for the need to model dynamic systems in astronomy, 
physics and geometry. The connection between the theory of DEs and linear and multilinear algebra is deep and multifaceted. The 
linear and multilinear algebra are foundational to the formulation, solution, and interpretation of DEs.  While DEs are primarily studied 
using functional analysis, linear and multilinear algebra do provide essential tools for solving, and analyzing these equations. A system of 
linear ordinary differential equations (ODEs) can be written as 
\[ \frac{d\bx}{dt} = A\bx, \]
where $\bx\in\RR^n$ is a vector and $A\in\RR^{n\times n}$ is a constant matrix. The solution to this equation involves some central concepts 
in linear algebra such as the eigenvalues, eigenvectors, and matrix exponentials $e^{At}$.  Now we may ask: what if $\bx$ or $t$ is (or both 
are) replaced by a matrix or a higher order tensor ?  

In this paper, we will define the derivative of a tensor (matrix) with respect to another tensor.  We focus on the tensor forms of derivatives 
of a matrix $X$ w.r.t. another matrix $T$.  This tensor form makes possible for us to unify the ODEs and partial differential equations 
(PDEs), and facilitates solution to them in some cases.  For our purpose, we first extend the outer product and contractive product of 
tensors and matrices to more general case through any partition of the modes, present some identities in terms of these products, we define 
the partial Tucker decompositions (TuckD) of a tensor, and use the partial TuckD to simplify the PDEs.  We also present a tensor form for the Lyapunov function. Our results in the products of tensors and matrices help us to establish some important equalities on the derivatives of matrices and tensors.    

\subsection{Background on Tensors}
\label{sec1-1:backgrd}

Tensors, as the central concept in multilinear algebra, are frequently used to describe PDEs in curvilinear coordinates. Tensor analysis is used 
to handle nonlinear terms in PDEs (e.g., $u\nabla u$, in fluid dynamics), which are inherently multilinear.  Solutions to PDEs (e.g., heat/wave equations) are expressed as series of eigenfunctions of differential operators (e.g., $\Delta u=\la u$), analogous to diagonalizing matrices in 
linear algebra.  While tensors are not central to the theory of differential equations, they become powerful for ODEs on manifolds. For partial 
differential equations (PDEs), tensors are indispensable tools for formulating covariant, geometrically meaningful, and computationally 
structured equations. To describe the motion constrained to some manifolds (e.g., rigid body rotation, robotic arms), we need some tools e.g. Riemannian metric tensors (matrices) 
to define distances or inner products which is crucial for kinetic energy, and the curvature tensor (third order tensor) governs the manifold's 
geometry which the ODE solution respects.  Also in the advanced geometric theory of ODEs, higher order derivatives, often described in tensor 
form, are treated as coordinates on a manifold, and the geometric structures on the jet spaces are also expressed by tensors. 

The concept of tensor can be dated back to the 19th century when Cauchy (1822) developed the stress tensor (second order tensor, i.e., a matrix), 
to describe internal forces in materials. Small order tensors (vectors and matrices) were formalized by Cauchy, Grassmann, and Hamilton in the 
mid-19th century. Grassmann (1844) introduced the idea of multilinear algebra which was the first explicit use of tensor-like objects in physics. 
Tensors are fundamental and ubiquitous in data sciences, mechanics, physics and chemistry, they are also used in PDEs especially those arising in continuum physics and geometry. For more detail on the development of tensor theory, we refer to \cite{QCC2018}.   

An essential power of tensors lies in its invariance of the algebraic form under coordinate transformations and the direct representation of 
physical quantities and geometric structures independent of coordinate systems.  For example, the Cauchy stress tensor (matrix) relates surface 
forces to directions, with governing equations inherently containing tensor derivatives, and strain tensors can be used to measure deformations 
(symmetric matrices).  In electromagnetism (Maxwell's equations), the E-M Field tensor, as a second order antisymmetric tensor (matrix), unifies 
electric (E) and magnetic (B) fields. In Einstein's general relativity, the equation $G_{uv}= (8\pi G/c^{4})T_{uv}$ relates the curvature of 
spacetime (described by Einstein tensor $G_{uv}$ derived from the 4-order Riemann curvature tensor) to the distribution of matter/energy. 

Tensors can be used to model order reduction of high-dimensional PDEs arising in quantum chemistry, stochastic PDEs, etc., to mitigate the 
``curse of dimensionality''. To make some algorithms more efficient, we usually use tensor algebra libraries to exploit inherent 
structure (symmetry, sparsity) for efficient storage and computation.
 
The term tensor was put in use early in 1837 by physicists and has been popular since 1925 when Albert Einstein used it to describe general 
relativity.  It can be found in chemometric, data science, image analysis, medical science, psychology and quantum physics, etc..  A tensor can 
be regarded as an extension of matrices.  The main topics in tensor analysis include tensor decompositions, spectral theory, nonnegative 
tensors,symmetric tensors\cite{cglm08,qi2013,qieig2005} and structured tensors \cite{qi2013,qieig2005,Xu2015}.  Even though tensors can be 
found in many areas, its appearance in ODE systems is rare if not missing.  Tensors are coordinate-invariant, making them ideal 
for modeling phenomena in physics (relativity, continuum mechanics), engineering (stress analysis, fluid dynamics), machine learning (neural 
networks, data representation), computer graphics (lighting, deformations),  quantum mechanics (spinors, entanglement).  Tensors obey strict 
rules when coordinates change (covariant/contravariant behavior). They describe relationships across multiple dimensions simultaneously.   

In general relativity, the metric tensor and stress-energy tensor are used to describe spacetime curvature and encodes mass-energy distribution 
respectively. In continuum mechanics,  the Cauchy stress tensor and the strain tensors are utilized to model forces in materials and describe 
material deformation respectively.  In electromagnetism (Maxwell's equations), the electromagnetic field tensor can be used to unifies E and B 
fields.  In machine learning and data science, weights in a neural network are stored as tensors, and data training can be speed up by tensor 
operations e.g. batch processing. In Natural Language Processing, tensors are used to represent semantic relationships in word embeddings. In 
computer graphics, the BRDF tensor models how light reflects off surfaces.  In quantum mechanics, we use the Pauli matrices (2D tensors) to 
describe electron spin and use density matrices to model mixed quantum states. 
  
In this paper, we first introduce some basic knowledge of tensors, including some basic terminology related to tensors and outer 
(tensor) products of tensors or matrices, the contractive product of tensors (matrices or vectors), and derivatives of matrices in tensor forms. 
Some properties on these products and derivatives are also presented. We also use tensors to express high order linear ordinary differential 
equations and present their solutions in tensor form.  For our purpose, we extend the outer product and contractive product of tensors (matrices) 
to a more general case through any partition of the modes, present some identities in terms of these products, initialize the definition of partial 
Tucker decompositions (TuckD) of a tensor, and use the partial TuckD to simplify the PDEs.  We also present a tensor form for the Lyapunov 
function. Our results in the products of tensors and matrices help us to establish some important equalities on the derivatives of matrices and 
tensors.  An algorithm based on the partial Tucker decompositions (TuckD) to solve the PDEs is given, and a numerical example is presented 
to illustrate the efficiency of the algorithm in the last section. \\  

\subsection{Organization of the Paper}
\label{sec1-2:organization}

The remainder of this paper is organized as follows.  

Section~\ref{sec2} introduces the preliminary concepts and notations for tensors that are used throughout the paper.  We recall the definitions of tensor unfoldings, mode-wise contractions, outer products, and the $k$-contractive product.  Special attention is given to the identity tensor 
$\caI_{2d;n}$ and the commutation tensor $\caK_{m,n}$, which play a central role in the tensor formulation of matrix derivatives.  Several lemmas establishing the associativity and commutativity properties of mixed contractive products are proved, laying the algebraic foundation for later developments.

Section~\ref{sec3} develops a systematic tensor calculus for matrix-valued functions.  We define the derivative of a matrix $Y$ with respect to a 
matrix $X$ as a fourth-order tensor $\frac{dY}{dX}\in\TT_{4}$, and derive closed-form expressions for the derivatives of fundamental matrix 
operations including products, inverses, powers, etc.  The case of symmetric matrix arguments is treated separately, where we introduce the 
symmetrized tensor $\X_s$ and the paired symmetric tensor formalism to handle the loss of independence among entries.

Section~\ref{sec4} recasts linear ordinary differential equations (ODEs) and partial differential equations (PDEs) into a unified tensor-matrix form.  We show that a multivariate $n$th-order linear ODE system can be equivalently written as $\frac{dX}{dt} = \A\ast X$, where $\A$ is a fourth-order coefficient tensor (C-tensor) and $X$ is a matrix whose columns collect successive derivatives of the state vector.  The general solution is expressed as $X = \exp(t\A)\ast C$, extending the matrix exponential to the tensor setting.  The extension to PDEs is discussed, and a sufficient condition for the solution to admit an exponential form is given in terms of the commutativity of the associated matrix slices.

Section~\ref{sec5} introduces the \emph{partial Tucker decomposition} (parTuckerD) as a flexible model reduction technique for high-order tensors.  Unlike the full Tucker decomposition which compresses all modes, parTuckerD compresses only a prescribed subset $\caM$ of modes, leaving the remaining modes intact.  This is particularly advantageous when certain modes carry essential structural or semantic information that must be preserved exactly.  We present Algorithm~\ref{alg5-1:partuckD-rand}, which computes the parTuckerD efficiently via randomized SVD, and provide a complexity analysis in Corollary~\ref{cor:complexity}.  The key theoretical result (Theorem~\ref{th5-4-1:ode}) shows that if the system tensor $\A$ admits a parTuckerD, then the tensor ODE $\frac{d\X}{d\T} = \A\ast \X$ can be reduced to a smaller system $\frac{d\ttX}{d\T} = \tG\ast \ttX$ operating entirely in the compressed subspace, yielding dramatic computational savings.

Section~\ref{sec6} presents comprehensive numerical examples that validate the theory and demonstrate the practical advantages of the proposed framework.  Example~\ref{exm6-1:partuck} uses a synthetic electronic health record (EHR) tensor of size $500\times 40\times 168\times 10$ to compare parTuckerD against the full Tucker decomposition. The results show that parTuckerD achieves essentially identical accuracy ($9.92\times 10^{-2}$ vs.\ $9.95\times 10^{-2}$) while perfectly preserving patient-level similarity structure (Spearman $\rho=1.0000$ vs.\ $0.9496$), at a modest compression ratio of $136.3\times$ compared to $4\,402.5\times$ for the full decomposition. Example~\ref{exm6-2} considers a sixth-order sparse tensor with $\caM=\{5,6\}$ and demonstrates that the rank-3 parTuckerD captures the dominant dynamics with a relative error on the order of $10^{-2}$. Example~\ref{exm6-3} applies the parTuckerD-based ODE solver to a $9\times 9$ matrix-valued system, confirming that the reduced integration in the $3\times 3$ subspace faithfully reproduces the full-space dynamics.  Section~\ref{sec6} also provides an error analysis that quantifies the sources of approximation error in the parTuckerD framework.  We decompose the total error into a \emph{decomposition error} arising from the low-rank truncation of $\A$ and a \emph{time-stepping error} arising from the numerical integration of the reduced ODE.  Theoretical bounds are derived and validated against the numerical experiments, showing that the decomposition error dominates and is accurately predicted by the tail singular values of the unfolded modes.

\vskip 5pt
\section{Preliminary on tensors}
\label{sec2}
\setcounter{equation}{0}

For any positive integers $m,n\colon 1\le m < n$, we denote throughout the paper by $[m,n]$ the set $\set{m, m+1, \ldots, n}$, and let $[n]$ denote 
$[1,n] =\set{1,2,\ldots, n}$, and $[n]^{m}$ the $m$-ary Cartesian product of the $m$ copies of set $[n]$,  $\RR$($\C$) the field of real (complex) numbers, $\RR^{n}$ ($\C^{n}$) the $n$-dimensional vector space on $\RR$($\C$), and $\RR^{m\times n}$ ($\C^{m\times n}$). 

In this paper, we frequently utilize the \emph{Kronecker delta} $\delta_{ij}$, which is defined as a function taking values in $\{0,1\}$ such that 
$\delta_{ij} = 1$ if and only if $i = j$.  For any positive integer $p$, we use $\Sym_{p}$ to denote the group of all permutations on set $[p]$.  
An $m$th order or $m$-order tensor $\A$ is a multiway array whose entries are denoted by $A_{i_1\ldots i_m}$ or $A_{\si}$ if 
$\si=(i_1,\ldots, i_m)$.  

An $m$th-order tensor is also called an $m$-order tensor for our convenience.  A tensor is called an \mnrt\ tensor if its dimensionalities of all 
modes are $n$.  Denote by $\TT_{m}$ the set of all $m$-order tensors, and $\TT_{\rmI}$ the set of $m$-order tensors indexed by 
$\rmI:= n_1\times\ldots \times n_m$, and $\TT_{m;n}$ the set of all  \mnts.  A tensor $\A$ is called \emph{symmetric} if each entry is invariant 
under all permutations on its indices.  We may use respectively  
\[
\bfm:=m_{1}\times \ldots \times m_{p}\quad \texttt{and} \quad \bfn:=n_{1}\times \ldots \times n_{q}
\]
to denote the sizes of an $p$-order tensor $\A$ and an $q$-order tensor $\B$.  Then the size of $\A\times \B$, the outer (tensor) product of $\A$ 
and $\B$, can be written as 
\[
\bfm\times \bfn=m_{1}\times \ldots \times m_{p}\times n_{1}\times \ldots \times n_{q}. 
\]

In the following, we present some basic concepts related to tensors, including some special tensors, e.g. positive (semi-)definite tensor, zero tensor, 
scalar tensor, identity tensor, hypercube tensor, rank-1 tensor; some operations such as the vectorization, extensive outer product and contractive 
product of matrices and tensors, and some decompositions such as the rank-1 decomposition (CP decomposition) and Tucker decomposition. Also 
introduced are some basic facts and results on these fundamental concepts.  

\subsection{Some basic on tensors}
\label{sec2-1}
\setcounter{equation}{0}  
 
An \mnst $\A$ corresponds to an $m$-order homogeneous polynomial $f_{\A}(\bx)$ described as
\beq\label{eq2-1}
f_{\A}(\bx):=\A\bx^m=\sum\limits_{i_1,i_2,\ldots,i_m} A_{i_1i_2\ldots i_m}x_{i_1}x_{i_2}\ldots x_{i_m}
\eeq 
A real symmetric tensor $\A$ is called \emph{positive definite} or \emph{pd} (\emph{positive semidefinite}, \emph{psd}) if  $f_{\A}(\bx)>0$ 
($f_{\A}(\bx)\ge 0$) for all  nonzero vector $\bx\in\RR^n$.  

Given a matrix $X\in\C^{m\times n}$.  The \emph{vectorization} of $X$ is defined as 
\[
\vecc(X)=(x_{11},x_{21},\ldots,x_{m1},\ldots, x_{1n},\ldots,x_{mn})^{\top}\in\C^{mn}
\]
Thus the linear space $\C^{m\times n}$ is isometric to $\C^{mn}$.  If $X$ is symmetric ($m=n$), we consider the \emph{patterned vectorization}
of $X$, denoted $\vecc_{s}(X)$, which is a vector of length $(n+1)n/2$, i.e., a subvector of $\vecc(X)$ obtained by the removal of duplicate entries from 
$\vecc(X)$.  For example, if  $X\in\C^{3\times 3}$, then   
\[
\vecc_{s}(X)= (x_{11},x_{12},x_{22},x_{13},x_{23},x_{33})^{\top}
\]
In this paper, we primarily focus on low-order tensors (i.e., tensors of order less than 6), but sometimes we also need to consider the general case.  
It is known that any $m$th-order tensor can be unfolded along each mode into an $(m-1)$th-order tensor. This unfolding process can be applied 
iteratively until a matrix or a vector is obtained. For example, a third-order tensor $\A$ 
of size $m \times n \times p$ can be unfolded into an $m \times np$ matrix, denoted $\mathbf{A}_{[3]}$, along mode-3 (with $\mathbf{A}_{[1]}$ and $
\mathbf{A}_{[2]}$ defined analogously). For a 4-order tensor $\A$ of size $m \times n \times p \times r$, there are more ways to reshape it into a matrix: it 
can be flattened into an $mn \times pr$ matrix $\mathbf{B}$ (or alternatively, $mp \times nr$, $mr \times np$, etc.) by grouping the first two modes as rows 
and the remaining two as columns, or into an $m \times npr$ matrix (or $n \times mpr$, etc.) by grouping all but one mode. 

Throughout the paper we use $X^{\top}$ to denote the transpose of a matrix or a vector $X$. Certain special types of matrices have natural extensions in 
tensor form. A \emph{zero tensor} is a tensor with all entries equal to 0, and an \emph{all-one tensor} is a tensor with all entries equal to 1. A \emph{diagonal element} of an $m$-order tensor $\A$ of size $[n]^{m}$ is an entry indexed as 
$A_{ii\ldots i}$, where $i\in [n]$. The tensor is referred to as a \emph{diagonal tensor} if all its off-diagonal entries are zero. Furthermore, a diagonal tensor 
$\A$ is called a \emph{scalar tensor} if all its diagonal elements are equal. For any $k\in [m]$, a \emph{slice} of an $m$-order tensor $\A$ along mode $k$ is 
an $(m-1)$-order tensor obtained by fixing the $k$th index. For example, a slice along mode 3 of an $m \times n \times p$ tensor $\A$ is a matrix 
$A(:,:,k)\in \C^{m \times n}$ for some $k \in [p]$, and a slice of an 4-order tensor is a third order tensor.

Given a vector $\bx=(x_1,\ldots, x_n)^{\top}$. We use $\bx^{m}$ to denote the symmetric rank-1 tensor defined by 
\[ 
\bx^{m}_{\si}=x_{i_1}x_{i_2}\ldots x_{i_m},\forall \si=(i_1,i_2,\ldots,i_m)\in S(m,n) \] 
It is known\cite{cglm08} that a real tensor $\A$ of size $n_1\times\ldots \times n_m$ can be decomposed as
\beq\label{eq2-4: cpd}
\A=\sum\limits_{j=1}^r \al_1^{(j)}\times \al_2^{(j)}\times \ldots \times \al_m^{(j)}
\eeq 
where $\al_i^{(j)}\in\C^{n_i}$ for $j\in [r],i\in [m]$. The smallest $r$ is called the rank of $\A$. (\ref{eq2-4: cpd}) is called a \emph{CP 
decomposition} or \emph{CPD} of $\A$.  It is called a \emph{symmetric CPD} if  there exists some vectors $\al^{(j)}\in\RR^{n}$ 
($j\in [r]$) such that (\ref{eq2-4: cpd}) holds if we take $\al^{(j)}=\al_1^{(j)}=\ldots =\al_m^{(j)}$ for all $j\in [r]$, i.e., 
\beq\label{eq2-5: symcpd}
\A = \sum\limits_{j=1}^r \be_j^{[m]} 
\eeq 
It is shown that $\A$ has a symmetric CPD if and only if $\A$ is a symmetric tensor\cite{cglm08}.\\  

\begin{lem}\label{le2-1}
Let $\A\in\TT_{m;n}$ be a symmetric tensor and $f_{\A}(\bx)$ be the $m$-order homogeneous polynomial associated with $\A$. Then
the derivative of $f_{\A}$ can be expressed as 
\beq\label{eq2-6: polyder}
 \frac{d f_{\A}(\bx)}{d \bx} = m \A\bx^{m-1}
\eeq  
\end{lem}
We note that $\A\bx^{m-1}$ on the right side of (\ref{eq2-6: polyder}) is defined by 
\beq\label{eq2-7:y=Ax^{m-1}}
(\A\bx^{m-1})_{i} = \sum_{i_{2},\ldots,i_{m}} A_{i i_{2}\ldots i_{m}}x_{i_{2}}\ldots x_{i_{m}}, \forall i\in [n]  
\eeq
which is the contractive product between $\A$ and tensor $\bx^{m-1}$ along the last $m-1$ modes of $\A$, and thus yields a vector in $\RR^{n}$.  

The proof of Lemma \ref{le2-1} can be found in \cite{QL2017}. As a corollary, we have  
\beq\label{eq2-7: quadder}
\frac{d f_{B}(\bx)}{d\bx} = 2B\bx
\eeq
for any square symmetric matrix $B\in\RR^{n\times n}$ and $\bx\in\RR^{n}$. 

\subsection{Extension of outer products and contractive products}
\label{sec2-2}
\setcounter{equation}{0} 

We now consider the set of 4-order tensors of size $m\times n\times m\times n$, denoted $\T[m,n]$, in which the contractive product is defined by 
\beq\label{eq2-8: ast-prod}
(\A\ast\B)_{i_{1}i_{2}i_{3}i_{4}} = \sum_{i_{1}^{\p},i_{2}^{\p}} A_{i_{1}i_{2}i_{1}^{\p}i_{2}^{\p}}B_{i_{1}^{\p}i_{2}^{\p}i_{3}i_{4}}
\eeq
The product defined by (\ref{eq2-8: ast-prod}) is called the \emph{2-contractive} (2C) product of  $\A$ and $\B$, and is also 
written as $\A\ast_{(3,4)}\B$ in \cite{XHL2018}.  We can see that $\T[m,n]$ is closed under the 2C product, and the associative law of 
the product can be verified by definition. \par 
\indent  Given any tensor $\A\in\T[m,n]$, we may define the powers of  $\A$ by induction as 
\beq\label{eq:tpower4} 
\A^{1}:=\A, \A^{2}=\A\ast \A,  \A^{k+1}:=\A\ast \A^{k}=\A^{k}\ast \A, \quad \forall k=1,2,\ldots, 
\eeq
and accordingly any polynomial of $\A$ can be defined.  Some elementary functions such as the exponential and log function can also be defined 
on $\T[m,n]$. In particular, the exponential function $\exp(\la \A)$ can be defined by
\beq\label{eq2-8: expA}
\sum_{k=0}^{\infty} \frac{\la^{k}}{k!}\A^{k} 
\eeq
The \emph{2C product} between an 4-order tensor and a compatible matrix can also be defined as 
\[ (\A B)_{ij} = \sum\limits_{k,l} A_{ijkl}B_{kl}, \forall i, j. \]
where $\A\in\RR^{m\times n\times p\times q}, B\in \RR^{p\times q}$. Similarly we can also define $C\ast \A$ as 
\[ (C\ast \A)_{ij} = \sum\limits_{k,l} C_{kl}A_{klij}, \forall i, j \]
if  $C\in \RR^{m\times n}$.  The 2C product can be extended to any \emph{$k$-contractive} product. Let $\A\in\TT_{p}, \B\in\TT_{q}$. If there exist some subset $S\subset [p], T\subset [q]$ such that  $\abs{S}=\abs{T}$ and the $S$-modes of $\A$ are compatible with the $T$-modes 
of  $\B$, we may assume w.l.g.  that $S=\set{i_{1},\ldots, i_{k}}, T = \set{j_{1},\ldots, j_{k}}$ with 
$1\leq i_{1}<\ldots < i_{k}\leq p, 1\leq j_{1}<\ldots < j_{k}\leq q$ ($k\leq \min\set{p,q}$).  The compatibility of $(\A, \B)$ along mode pairs 
$(S,T)$ means that 
\[ n_{i_{1}}=m_{j_{1}}, n_{i_{2}}=m_{j_{2}}, \ldots, n_{i_{k}}=m_{j_{k}}, \]
where $\A$ and $\B$ are of size $n_{1}\times\ldots\times n_{p}$ and $m_{1}\times\ldots\times m_{q}$ respectively.  Then the contractive 
product $\A\ast_{(S,T)} \B$ yields an $(p+q-2k)$-order tensor.  A special case is when 
\[ S=\set{p-k+1, p-k+2, \ldots, p}, T =\set{1,2,\ldots, k}, \] 
i.e., $S$ consists of the last $k$ modes of $\A$, and $T$ consists of the first $k$ modes of  $\B$. Then we have 
\beq\label{eq:t-t-contract}
(\A\ast\B)_{i_1\ldots i_{p-k} j_{k+1}j_{k+2}\ldots j_q} = \sum_{i_{p-k+1},i_{p-k+2},\ldots, i_p} 
A_{i_1 i_{2}\ldots i_{p-k}i_{p-k+1}\ldots i_{p}} B_{i_{p-k+1}i_{p-k+2}\ldots i_{p}j_{k+1}\ldots j_q}
\eeq 
In this case, we denote $\A\ast_{(S,T)}\B$ simply by $\A\ast_{[k]} \B$. Furthermore, if  $k=q\leq p$, we denote it by $\A\ast \B$.  Note that 
when $k=p=q$ ($\A$ and $\B$ have the same size), we have $\A\ast\B = \seq{\A,\B}$, i.e., inner product of $\A$ and $\B$, which is the 
extension of two matrices.  Sometimes we may mix the notation $\ast$ for either cases whenever it makes sense. For example, in the expression 
$(\A\ast\B) \ast C$, $\A\ast\B$ may be defined as $\A\ast_{(S,T)}\B$ as defined above, and the contractive product between $\A\ast\B$ and $C$ 
($C$ is a matrix) should be understood as $(\A\ast\B)\ast_{(W,\set{1,2})} C$ where $W$ consists of the indices of last two modes of $\A\ast\B$. 
This is the contractive product commonly defined between tensor $\A$ of size $n_{1}\times\ldots \times n_{p}$ and a matrix $B\in\RR^{n_{k}\times m_{k}}$.  For any positive integer $k\in [p]$. We use $\A\ast_{k}B$ to denote the contractive product of $\A$ with $B$ along mode pair 
$(\set{k},\set{1})$. We call this kind of contractive product a \emph{1M} contractive product.  Analoguously we can also define $C\ast_{k}\A$ 
if $C\in\RR^{m_{k}\times n_{k}}$. When $B\in\RR^{n_{k}\times n_{k}}$, it is easy to see that 
\beq\label{eq: 1ccontracprod}
 \A\ast_{k} B = B^{\top}\ast_{k} \A
\eeq
\begin{prop}\label{prop2-2:contractprop}
Let $\A\in\TT_{p}$ be a tensor of size $n_{1}\times\ldots\times n_{p}$ where $p\ge 2$, and $B$ and $C$ are matrices of appropriate sizes. 
Then we have 
\begin{description}
\item[(1)]  $\A\ast_{k} B\ast_{k} C =\A\ast_{k} (BC)$ if  $B\in\RR^{n_{k}\times m_{k}}, C\in\RR^{m_{k}\times l_{k}}$. 
\item[(2)] $\A\ast_{i} B\ast_{j}C=\A\ast_{\set{i,j}}(B\times_{c} C)$ if  $1\le i<j\le p$ and the row numbers of $B, C$ are resp. $n_{i}$ and $n_{j}$. 
\end{description}
\end{prop} 
\begin{proof}
To show the item (1), we first note that both sides yield $p$-order tensors of size 
\[ n_{1}\times\ldots n_{k-1}\times l_{k}\times n_{k+1}\times  \ldots \times n_{p}. \]
Furthermore, if we denote by $\bbI$ for the index set of $\A\ast_{k} B\ast_{k} C$ (also the index set of $\A\ast_{k} (BC)$), then for any 
$(i_{1},\ldots,i_{p})\in \bbI$, we have 
\beyy
(\A\ast_{k}B\ast_{k} C)_{i_{1}i_{2}\ldots i_{p}} 
&=& \sum_{i_{k}^{\p}}(\A\ast_{k}B)_{i_{1}\ldots i_{k-1}i_{k}^{\p}i_{k+1}\ldots i_{p}} C_{i_{k}^{\p}i_{k}}\\
&=& \sum_{i_{k}^{\p}}\left(\sum_{i_{k}^{\p\p}} A_{i_{1}\ldots i_{k-1}i_{k}^{\p\p}i_{k+1}\ldots i_{p}}
B_{i_{k}^{\p\p} i_{k}^{\p}} \right) C_{i_{k}^{\p}i_{k}}\\ 
&=&\sum_{i_{k}^{\p\p}} A_{i_{1}\ldots i_{k-1}i_{k}^{\p\p}i_{k+1}\ldots i_{p}} (BC)_{i_{k}^{\p\p}i_{k}}\\
& =&\left[ \A\ast_{k}(BC)\right]_{i_{1}i_{2}\ldots i_{p}} 
\eeyy
Thus (1) holds.  To prove (2), we may assume $i=1,j=2$ such that the index is not so complicate (yet the arguments should be similar).
Thus for any $(i_{1}, \ldots, i_{p})$, we have 
\beyy
(\A\ast_{1}B\ast_{2} C)_{i_{1}i_{2}\ldots i_{p}} 
&=&\sum_{i_{2}^{\p}}(\A\ast_{1}B)_{i_{1}i_{2}^{\p}i_{3}\ldots i_{p}}C_{i_{2}^{\p}i_{2}}\\
& =&\sum_{i_{2}^{\p}}\left( \sum_{i_{1}^{\p}} A_{i_{1}^{\p}i_{2}^{\p}i_{3}\ldots i_{p}} B_{i_{1}^{\p} i_{1}} \right) C_{i_{2}^{\p}i_{2}}\\
& =&\sum_{i_{1}^{\p},i_{2}^{\p}}A_{i_{1}^{\p}i_{2}^{\p}i_{3}\ldots i_{p}}\left( B_{i_{1}^{\p}i_{1}}C_{i_{2}^{\p}i_{2}}\right)\\
& =&\left[\A\ast_{\set{1,2}} (B\times_{c} C)\right]_{i_{1}i_{2}\ldots i_{p}}
\eeyy 
Thus $\A\ast_{1} B\ast_{2}C=\A\ast_{\set{1,2}}(B\times_{c} C)$. We can also show 
$\A\ast_{i} B\ast_{j}C=\A\ast_{\set{i,j}}(B\times_{c} C)$ for any $(i,j)\colon 1\le i<j\le p$. 
\end{proof}

From Proposition~\ref{prop2-2:contractprop}, we have  
\begin{cor}\label{co2-3}
Let $\A\in\TT_{p}$ with size $n_{1}\times\ldots\times n_{p}$, $k\in [p]$ and $U\in\RR^{n_{k}\times n_{k}}$ be an orthogonal matrix. Then 
\beq\label{eq:co2-3-1}
 \A\ast_{k} U\ast_{k} U^{\top} = U\ast_{k} \A \ast_{k} U = \A 
\eeq
\end{cor}

\begin{proof}
By (1) of Proposition~\ref{prop2-2:contractprop}, we have 
\[
 \A\ast_{k} U\ast_{k} U^{\top} = \A\ast_{k} (U U^{\top}) = \A\ast_{k} I_{n_{k}}.
\]
Since for any index $(i_{1},\ldots, i_{p})$
\[
\left( \A\ast_{k} I_{n_{k}}\right)_{i_{1}i_{2}\ldots i_{p}} = 
\sum_{i^{\p}_{k}} A_{i_{1}\ldots i^{\p}_{k}\ldots i_{p}} \delta_{i^{\p}_{k} i_{k}} = A_{i_{1}\ldots i_{k}\ldots i_{p}},
\]
therefore $\A\ast_{k} I_{n_{k}} = \A$ for any $k\in [p]$.  Similarly we can show $I_{n_{k}}\ast_{k} \A = \A$. Thus \eqref{eq:co2-3-1}
holds.
\end{proof}

Note that when $\A$ is a matrix $A\in\RR^{n\times n}$ and $U\in\RR^{n\times m}$, we have 
\[
U^{\top} A U = A\ast_{1} U\ast_{2} U = A\ast_{2} U\ast_{1} U \in\RR^{m\times m},
\]
which can be generalized as a transformation 
\[
\A\dsb{U} = \A\ast_{1} U \ast_{2} U \ldots \ast_{p} U \in\TT_{p;m}
\]
when $\A\in\TT_{p;n}$ and $U\in\RR^{n\times m}$. 

For any two tensors $\A$ and $\B$ with order $p$ and $q$ respectively, we define the \emph{outer product} (or \emph{tensor product}) of 
$\A$ and $\B$, denoted $\A\times \B$, as 
\beq\label{eq:outprod}
(\A\times\B)_{i_{1}\ldots i_{p}j_{1}\ldots j_{q}} = A_{i_{1}\ldots i_{p}}B_{j_{1}\ldots j_{q}}
\eeq
When $A, B$ are matrices of size resp. $m\times n$ and $p\times q$, the outer product $A\times B$ is an 4-order tensor 
of size $m\times n\times p\times q$.  The outer product can be extended in several different ways. For example, if $A$ and $B$ are two matrices, 
we can assign $(1,3)$-modes to $A$ and $(2,4)$-modes to $B$ to produce the 4-order tensor $A\times_{c} B$,i.e.,  
\beq\label{eq:crossprod}
(A\times_{c} B)_{i_{1}i_{2}i_{3}i_{4}} = A_{i_{1}i_{3}}B_{i_{2}i_{4}}
\eeq 
Similarly, if we assign $(1,4)$-modes to $A$ and $(2,3)$-modes to $B$, then we have the 4-order tensor $A\times_{ac} B$ defined as 
\beq\label{eq:acrossprod}
(A\times_{ac} B)_{i_{1}i_{2}i_{3}i_{4}} = A_{i_{1}i_{4}}B_{i_{2}i_{3}}
\eeq
We can also extend the outer products introduced above to more general case. Let $\A_{i}$ be tensor of order $p_{i}$ for $i\in [r], p=p_{1}+\ldots +p_{r}$, and 
\[ \pi:=\pi_{1}\cup\pi_{2}\cup \ldots \cup \pi_{r}     \]
be a partition of set $[p]$.  We denote by
\beq\label{eq:prod-partit01} 
[\A_{1}\times\ldots\times\A_{r}][\pi] 
\eeq
for the outer product 
\beq\label{eq:prod-partit02} 
\caK = \A_{1}\times_{\pi_{1}}\A_{2}\times_{\pi_{2}}\ldots \A_{r-1}\times_{\pi_{r-1}}\A_{r} 
\eeq
i.e., the modes in $\pi_{k}$ are assigned to $\A_{k}$ for each $k\in [r]$. For example, if $A, B, \C$ are resp. tensors of order $2,2,4$ 
($A,B$ are matrices), and $\pi=\set{\set{1,5},\set{2,6},\set{3,4,7,8}}$.  Then $\caK=[A\times B\times \C][\pi]$ is an 8-order tensor with components 
defined by
\[ K_{i_{1}\ldots i_{8}} = a_{i_{1}i_{5}}b_{i_{2}i_{6}}c_{i_{3}i_{4}i_{7}i_{8}} \] 

An $d$-order tensor $\caD=(D_{i_{1}\ldots i_{d}})\in\TT_{d}$ is called a \emph{hypercube} if the dimensionality on each mode is the same. 
Thus $\caD$ is a hypercube if it is of size $[n]^{d}$.  We denote by $\caD(\A)\in\TT_{d,n}$ the diagonal tensor with 
\[ D_{ii\ldots i}=A_{ii\ldots i}, \forall i\in [n].    \]
We call $\caD$ an \emph{unit tensor} if  $D_{ii\ldots i}=1$ for all $i$.  Denote by $\caJ_{d;n}$ (or $\caJ$ if no risk of confusion about its order and its 
dimension arises) the $d$-order $n$-dimension unit tensor ($\caJ_{2;n}= I_{n}$ is the $n\times n$ identity matrix).  It is taken for granted that $\caJ$ 
may correspond to an identity transformation in $\TT_{2d;n}$, i.e.,  $\A\ast\caJ =\caJ \ast \A= \A$ for each $\A\in\TT_{d;n}$.  Unfortunately this is not true.  
In fact, we have 
\begin{prop}\label{prop2-3:unit-t}
Let $\A\in\TT_{d;n}$. Then $\A\ast\caJ = \caJ \ast \A=\A$ if and only if  $\A$ is a diagonal tensor.
\end{prop} 
\begin{proof}
For any $(i_{1},\ldots,i_{d})\in [n]^{d}$, we have  
\beyy 
(\caI\ast \A)_{i_{1}\ldots i_{d}}
&=& \sum_{j_{1},\ldots,j_{d}} \delta_{i_{1}\ldots i_{d}j_{1}\ldots j_{d}} A_{j_{1}\ldots j_{d}}\\
&=& \delta_{i_{1}\ldots i_{d}i_{1}\ldots i_{d}} A_{i_{1}\ldots i_{d}}
\eeyy 
It follows that 
\[ (\caI\ast \A)_{i_{1}\ldots i_{d}} 
=\begin{cases} A_{ii\ldots i}, &\texttt{ if  } i_{1}=\ldots =i_{d}=i;\\  0, & \texttt{otherwise}. \end{cases}
\]
Thus $\caI\ast \A = \caD(\A) $.  Similarly we can show that $\A\ast\caI = \caD(\A) $. 
\end{proof}

Now we wish to define a tensor whose performance is similar to that of the identity matrix in the matrix case, i.e., mapping any tensor 
$\A\in\TT_{d;n}$ to itself.  For this purpose, we let 
\beq\label{eq:iden}  
\caI_{2d;n} = [I_{n}, \ldots, I_{n}][\pi]
\eeq
where $\pi=\set{\set{1,d+1},\set{2,d+2}, \ldots, \set{d, 2d}}$.  Denote $\caI:=\caI_{2d;n}\in\TT_{2d;n}$ for simplicity (if no risk of confusion arises). 
Then $\caI$ is the tensor generated by the outer product of  $d$ copies of $I_{n}$. We can show that 
\begin{lem}\label{le2-2}
\beq\label{eq:le2-2-1} 
\caI\ast \A = \A = \A\ast \caI, \forall  \A\in\TT_{d;n},
\eeq
where $\ast$ in the first equality is the contractive product of $\caI$ with $\A$ along last $d$ modes of $\caI$, and $\ast$ in the second is the 
contractive product of $\A$ with $\caI$ along the first $d$ modes of $\caI$. 
\end{lem}
\begin{proof}
Let $\B:=\caI\ast \A$.  We first notice that 
\beq\label{eq2-2-2}
\caI_{i_{1}\ldots i_{d};j_{1}\ldots j_{d}} = \delta_{i_{1}j_{1}}\ldots \delta_{i_{d}j_{d}} 
\eeq
for each $(i_{1},\ldots,i_{d};j_{1},\ldots, j_{d})\in [n]^{2d}$.  Thus for any $(i_{1},i_{2},\ldots,i_{d})\in [n]^{d}$, we have by (\ref{eq2-2-2}) that 
\beyy\label{eq2-2-3}
B_{i_{1}i_{2}\ldots i_{d}}
&=& \sum_{j_{1},\ldots, j_{d}} \caI_{i_{1}\ldots i_{d};j_{1}\ldots j_{d}} A_{j_{1}\ldots j_{d}}\\ 
&=&\sum_{j_{1},\ldots, j_{d}}\delta_{i_{1}j_{1}}\ldots \delta_{i_{d}j_{d}} A_{j_{1}\ldots j_{d}}\\
&=& A_{i_{1}\ldots i_{d}}
\eeyy
which implies that $\caI\ast \A = \A$ holds  for each $\A\in\TT_{d;n}$. Similarly we can show $\A\ast \caI = \A$.  Thus (\ref{eq:le2-2-1}) is proved. 
\end{proof}
 
The following lemmas show the associativity for any mixed contractive products of tensors.  
\begin{lem}\label{le4ode1}
For any tensor $\A\in\TT_{p+q},\B\in\TT_{q}$ and matrix $U$ of appropriate size such that all products involved make sense, we have   
\begin{description}
\item[(1)]  If $1\le k\le p$,
\beq\label{eq:AUB01}
(\A\ast_{k} U)\ast_{[q]} \B = \left[ \A \ast_{[q]} \B\right] \ast_{k} U
\eeq
\item[(2)]  If $p+1\le k\le p+q$, 
 \beq\label{eq:AUB02}
(\A\ast_{k} U)\ast_{[q]} \B = \A \ast_{[q]} (U\ast_{(k-p)}\B) = \A \ast_{[q]} (\B\ast_{(k-p)}U^{\top})  
\eeq
\end{description}
\end{lem}
\begin{proof}
To prove (1), we may assume $k=1$ and denote $\C = (\A\ast_{1} U)\ast_{[q]} \B$ which is an $p$-order tensor. Then for any index 
$(i_{1},\ldots,i_{p})$, we have 
\beyy
C_{i_{1}\ldots i_{p}} 
&=& \sum_{j_1,\ldots, j_q} (\A\ast_{1} U)_{i_1\ldots i_p j_1\ldots j_q} B_{j_1\ldots j_q} \\
&=& \sum_{j_1,\ldots, j_q} \left(\sum_{i_1^{\p}} A_{i_{1}^{\p}i_2\ldots i_p j_1\ldots j_q}u_{i_1^{\p}i_1}\right)B_{j_1\ldots j_q} \\
&=& \sum_{i_1^{\p}} \left(\A\ast_{[q]}\B\right)_{i_1^{\p}i_2\ldots i_p} u_{i_1^{\p}i_1} \\ 
&=& \left[\left(\A\ast_{[q]}\B\right)\ast_{1} U \right]_{i_1i_2\ldots i_p} 
\eeyy
Thus $(\A\ast_{1} U)\ast_{[q]} \B =\left(\A\ast_{[q]}\B\right)\ast_{1} U$. This argument applies to any $k\in [p]$, thus (\ref{eq:AUB01}) is 
proved.  

To prove (2), we denote $k^{\p}:=k-p$, then $1\le k^{\p}\le q$. Similar argument shows 
\[ (\A\ast_{k} U)\ast_{[q]} \B =\A \ast_{[q]} (\B\ast_{k^{\p}}U^{\top}) \]
Noting that $\B\ast_{k^{\p}}U^{\top}=U\ast_{k^{\p}}\B$ for each $k^{\p}\in [q]$, (\ref{eq:AUB02}) is proved.  
\end{proof} 

By Lemma~\ref{le4ode1}, we can prove  
\begin{lem}\label{le4ode2}
For any tensor $\A\in\TT_{p+d},\B\in\TT_{d}$ and $k\in \set{p+1,p+2, \ldots, p+d}$, we have 
\beq\label{eq:simplifyode}
\left[ (\A\ast_{k} U)\ast_{[d]} \B\right]\ast_{k} V = \A \ast_{[d]} \left[\B\ast_{k^{\p}} (U^{\top}V) \right] =\A \ast_{[d]} \left[(V^{\top}U)\ast_{k^{\p}} \B \right] 
\eeq
where $k^{\p}=k-p\in [d]$, $U,V$ are matrices of appropriate sizes such that all products involved in \eqref{eq:simplifyode} make sense. 
\end{lem}

\begin{proof}
It is easy to confirm by Lemma~\ref{le4ode1} that 
\beq\label{eq:le4ode01}  
\left[ (\A\ast_{k} U)\ast_{[d]} \B\right]\ast_{k} V = \A\ast_{[d]} \left[ U\ast_{k^{\p}} \B\ast_{k^{\p}} V \right]
\eeq 
So we need only to show that 
\beq\label{eq:simplifyode2}
U\ast_{k^{\p}} \B\ast_{k^{\p}} V = (V^{\top}U)\ast_{k^{\p}} \B = \B\ast_{k^{\p}} (VU^{\top})
\eeq
We take $k^{\p}=1$ for the convenience, and denote by $\caF$ and $\caG$ for the lhs and rhs of (\ref{eq:simplifyode2}) respectively. Then
\beyy
F_{i_{1}i_{2}\ldots i_{d}} 
&=& \sum_{i^{\p}_{1}} u_{i_{1}i^{\p}_{1}} \left( \sum_{i^{\p\p}_{1}}B_{i^{\p\p}_{1}i_{2}\ldots i_{d}}v_{i^{\p\p}_{1} i^{\p}_{1}}\right) \\
&=& \sum_{i^{\p\p}_{1}} B_{i^{\p\p}_{1}i_{2}\ldots i_{d}}  \left( u_{i_{1}i^{\p}_{1}} v_{i^{\p\p}_{1}i^{\p}_{1}} \right) \\
&=& \left[ \B\ast_{1} (VU^{\top})\right]_{i_{1}i_{2}\ldots i_{d}} 
\eeyy
Thus $U\ast_{1} \B\ast_{1} V = \B\ast_{1} (VU^{\top})$. Similar we can show that $U\ast_{1} \B\ast_{1} V = (V^{\top}U)\ast_{1} \B$. 
Thus (\ref{eq:simplifyode2}) holds for $k^{\p}=1$.  The same argument goes for any $k^{\p}\in [d]$.  Consequently (\ref{eq:simplifyode}) follows immediately from (\ref{eq:simplifyode2}) and (\ref{eq:le4ode01}).  
\end{proof}   

The contractive product can be utilized to express the Tucker decompositions (TuckDs).  The Tucker decomposition of an 3-order tensor was
first defined by Tucker in 1966\cite{Tuck1966}. 

\subsection{Tucker decompositions}
\label{sec2-3}
\setcounter{equation}{0} 

Given an $m$-order tensor $\A$ of size $n_1\times\ldots\times n_m$.  A \emph{Tucker decomposition}, or called \emph{full Tucker 
decomposition}, of $\A$ is in form 
\beq\label{eq2-3-1:tuckD}
\A = \caG\ast_{1}U_{1}\ldots \ast_{m}U_{m}
\eeq
where $U_{k}\in\RR^{r_{k}\times n_{k}}$, called \emph{factor matrix}, satisfies $U_{k}U_{k}^{\top} = I_{r_{k}}$, $r_k\ll n_k (k\in [m])$, 
and $\caG$, the \emph{core tensor} (of size $r_1\times\ldots \times r_m$) satisfies 
\beyy\label{eq:coret}
 \norm{G(i_{1},:,:,\ldots,:)}_{F} =\si_{i_{1}}(G[1]), \quad  &\forall i_{1}\in [n_{1}]. \\
 \norm{G(:,i_{2},:,\ldots,:)}_{F} =\si_{i_{2}}(G[2]), \quad  &\forall i_{2}\in [n_{2}]. \\
  \cdots  \quad   \cdots  \quad  \cdots  \quad   &  \cdots  \\
 \norm{G(:,:,\cdots,:, i_{m})}_{F}=\si_{i_{m}}(G[m]), \quad  &\forall i_{m}\in [n_{m}], 
\eeyy    
where $G[k]$ is the unfolding matrix of $\caG$ along mode $k$ and $\si_{i}(M)$ denotes the $i$th singular value of a matrix $M$.  The 
\emph{Tucker rank} $(r_{1},\ldots, r_{m})$ is a $d$-tuple with $r_{k}=\rank(A[k])$.   

Now we let $\pi=\set{\set{1,m}, \set{2,m+1}, \ldots, \set{m, 2m}}$ be a partition of set $[2m]$ and let 
\beq\label{eq2-3-2:U-pi01}
\U:=\dsb{U_{1},\cdots,U_{m}}_{\pi} 
\eeq
Then $\U$, as the outer product of $\set{U_{j}}$ in terms of partition $\pi$, is a tensor in $\TT_{2m}$ of size 
\[ r_{1}\times\ldots\times r_{d}\times n_{1}\times\ldots \times n_{m}, \]  
that is, 
\beq\label{eq:U-pi02}
\dsb{U_{1},\cdots,U_{m}}_{\pi} = U_{1}\times_{\al_{1}}U_{2}\times_{\al_{2}} \cdots \times_{\al_{m}}U_{m} 
\eeq
where $\pi=\set{\al_{1}, \ldots,\al_{m}}$ is a 2-partition of set $[2m]$ with $\al_k = \set{k, m+k}$ for all $k\in [m]$.  The 
TuckD~(\ref{eq2-3-1:tuckD}) can also be represented as 
\beq\label{eq2-3-1:tuckD2}  
\A = \caG\ast_{[m]}\U 
\eeq
Note that each $U_{k}$ can be computed by the SVD of $A[k]$, i.e., $A[k]=U_{k}\Sigma_{k}V_{k}^{\top}$, and $\caG$ can be obtained by 
\[ \caG = \A\ast_{1} U_{1}^{\top}\ldots \ast_m U_{m}^{\top}. \]

In Section \ref{sec5}, we will extend the Tucker decomposition to the \emph{partial Tucker decomposition}.  

By Lemma \ref{le4ode2} we have 
\begin{cor}\label{co2-3-1: TD}
Let $\A\in\TT_{p+q},\B\in\TT_{q}$ and $\A$ has a Tucker decomposition  
\beq\label{eq:tuckdec3}
\A = \caG\ast_{1}U_{1}\ldots \ast_{p}U_{p}\ast_{p+1} V_{1}\ast_{p+2} V_{2}\ldots \ast_{p+q} V_{q}
\eeq
where $U_{k}\in\RR^{r_{k}\times n_{k}}$ ($r_{k}\le n_{k}$). If $p\le q$, then we have 
\beq
\A\ast_{[q]} \B = \caG\ast_{[q]} (\U\ast_{[p]} \B\ast \V^{\top} ) 
\eeq
where $\U:=U_{1}\times \ldots \times U_{p}, \V:=V_{1}\times \ldots \times V_{q}$.  Note that $\U$ and $\V$ are tensors of 
order $2p$ and $2q$, and the tensor $\U\ast_{[p]} \B\ast \V^{\top}\in\TT_{q}$. 
\end{cor}

Note that in Corollary \ref{co2-3-1: TD}, if $p>q$, then we can choose any $q$-subset $\bbW\subset [p]$ and let 
$\U=[U_{i_{1}},\ldots,U_{i_{q}}]$ where $\bbW=\set{i_{1},i_{2},\ldots,i_{q}}$, and replace  (\ref{eq:tuckdec3}) by the $\bbW$-TD.     

Now we consider set $\T[m,n]$.  Since $\T[m,n]$ is closed under the 2C product defined by (\ref{eq2-8: ast-prod}), we can not only 
 define the identity tensor and even the inverse for any tensor in $\T[m,n]$.

\subsection{Tensor products and Identity tensors}\label{sec2-4}
\setcounter{equation}{0} 
 
We define an 4-order tensor $\caI_{m,n} \in \T[m,n]$ by  
\beq\label{eq2-4-1:iden}
I_{ijkl} = \delta_{ik} \delta_{jl},
\eeq  
where $\delta_{ij}$ denotes the Kronecker delta, i.e., $\delta_{ij}=1$ if and only if $i=j$ and $\delta_{ij}=0$ otherwise. We call tensor 
$\caI_{m,n}$ an \emph{identity tensor} in $\T[m,n]$. By definition, we have 
\beq\label{eq2-4-2:iden}
\caI_{m,n} = I_m \times_c I_n
\eeq
and it satisfies  
\beq\label{eq2-4-3} 
\caI\ast A = A = A\ast \caI, \forall  A\in\RR^{m\times n}
\eeq
which justifies the nomenclature. We have the following equalities       
\begin{lem}\label{le2-4-1}
Let $A\in\RR^{m\times n}$ be a matrix.  Then we have 
\begin{description}
\item[(1)] $A^{\top}\ast_1 (I_m\times_{c} I_n)=(I_m\times_c I_n)\ast_1 A = A^{\top}\times_{c} I_n$.
\item[(2)] $A\ast_2 (I_m\times_{c} I_n)=(I_m\times_c I_n)\ast_2 A^{\top} = I_{m}\times_{c}A$.
\item[(3)] $A^{\top}\ast_3 (I_m\times_{c} I_n)=(I_m\times_c I_n)\ast_3 A=A\times_{c} I_n$.
\item[(4)] $A\ast_4 (I_m\times_{c} I_n)=(I_m\times_c I_n)\ast_4 A^{\top} = I_{m}\times_{c}A^{\top}$.
\end{description}
\end{lem}  
\begin{proof}
We present a proof to (4), and other items can be proved similarly.  First notice that the left, right and the middle part of  (4) are all 
4-order tensors of size 
$m\times n\times m\times m$. Furthermore,  for any index $(i,j,k,l)\in [m]\times [n]\times [m]\times [m]$, we have 
\beyy
\left[(I_{m}\times_{c} I_{n})\ast_{4} A^{\top}\right]_{ijkl} 
& =& \sum_{l^{\p}}(I_{m}\times_{c} I_{n})_{ijkl^{\p}}A_{ll^{\p}} \\
& =& \sum_{l^{\p}}\delta_{ik} \delta_{jl^{\p}}A_{ll^{\p}} =\delta_{ik} A_{lj}\\
& =& (I_{m}\times_{c} A^{\top})_{ijkl}    
\eeyy 
Thus, the second equality in (4) holds. Similar reasoning applies to show that 
\[ A\ast_4 (I_m\times_{c} I_n) = I_{m}\times_{c}A^{\top}. \] 
Therefore, (4) is verified. The remaining items can be established analogously.
\end{proof}
 
Tensor $\caK_{m,n}:= I_{m}\times_{ac} I_{n}$ is called the \emph{commutation tensor} originated from commutation matrix $K_{m,n}$. It
transforms a matrix $A\in\bbC^{m\times n}$ into its transpose $A^{\top}$\cite{XHL2018}.  The powers of a tensor $\A\in\T[m,n]$ are defined 
by (\ref{eq:tpower4}) since the associative law holds in $\T[m,n]$,  $\A^{k}$ is well-defined for any tensor $\A\in\T[m,n]$ and any positive integer $k$.  

\begin{lem}\label{le2-4}
Let $\A\in\T[m,n]$ and $B\in\RR^{m\times n}$. Then
\begin{description}
\item[(1)]  $\A^{k+1}\ast B =\A\ast (\A^{k}\ast B)$;
\item[(2)]  Let $t\in\RR$ be a variable and $\caF(t)=\exp(t\A)$.  Then $\frac{d\caF}{dt}=\A\ast \caF(t)$. 
\end{description}
\end{lem}

\begin{proof}
To prove (1), we first let $k=0$. Then it is equivalent to 
\[ \A\ast B = \A\ast (\caI\ast B) \]
which follows directly from (\ref{eq2-4-3}).   For case $k=1$, we have for any index $(i,j)\in [m]\times [n]$ that 
\beyy
(\A^2\ast B)_{ij} &=& \sum_{kl} (\A^{2})_{ijkl}B_{kl}\\
                             &=& \sum_{kl} \left(\sum_{i^{\p},j^{\p}} A_{iji^{\p}j^{\p}}A_{i^{\p}j^{\p}kl}\right) B_{kl}\\
                             &=& \sum_{i^{\p},j^{\p}} A_{iji^{\p}j^{\p}} \left(\sum_{kl} A_{i^{\p}j^{\p}kl}B_{kl}\right) \\
                             &=& \sum_{i^{\p},j^{\p}} A_{iji^{\p}j^{\p}} (\A\ast B)_{i^{\p}j^{\p}}\\
                             &=& \left[ \A\ast (\A\ast B)\right ]_{ij}
\eeyy
Thus $\A^{2}\ast B =\A\ast (\A\ast B)$. By the induction to $k$ we can show (1).  In fact, (1) can be verified by the associativity  
\[ (\A\ast \B)\ast \C = \A \ast (\B\ast \C)  \]
which holds for any compatible tensors (matrices) $\A, \B$ and $\C$. \par 
\indent  To show item (2), we notice the series expansion 
(\ref{eq2-8: expA})  and 
\[
\frac{d}{dt} e^{t\A} =\frac{d}{dt} \sum_{k=0}^{\infty} \frac{t^{k}}{k!} \A^{k}
=\sum_{k=0}^{\infty} \frac{d}{dt}\left(\frac{t^{k}}{k!}\right) \A^{k}
=\sum_{k=1}^{\infty} \frac{t^{k-1}}{(k-1)!}\A^{k}=\A\ast e^{t\A}.       
\]
\end{proof}

For any matrix $A\in\RR^{n\times n}$, we define 
\beq\label{eq2-3-1} 
\A^{c} = I_{n}\times_{c} A + A\times_{c} I_{n}  
\eeq
and   
\beq\label{eq2-3-2} 
\A^{ac} = I_{n}\times_{ac} A +A \times_{ac} I_{n}
\eeq
Then $\A^{c}$ and $\A^{ac}$ are both 4-order tensors of size $n\times n\times n\times n$.  We may regard $\A^{c}$ and $\A^{ac}$ as the 
transformations on $\C^{n\times n}$.  
\begin{thm}\label{th:2-5}
Let $A\in\C^{n\times n}$ be a matrix.  For any $X\in\C^{n\times n}$, we have 
\begin{description}
\item[(i)]  $\A^{c}\ast X = AX+XA^{\top}$. 
\item[(ii)]  $ \A^{ac}\ast X = AX+X^{\top}A$. 
\end{description}
\end{thm}

\begin{proof}
To prove (i), we let $(i,j)\in [n]^{2}$. Then 
\[
\left[(I_{n}\times_{c} A)\ast X\right]_{ij}=\sum_{k,l} (I_{n}\times_{c}A)_{ijkl} X_{kl} = \sum_{k,l} \delta_{ik} A_{jl} x_{kl}
=\sum_{l} A_{jl}x_{il}= (XA^{\top})_{ij}.
\]
So $(I_{n}\times_{c} A)\ast X = XA^{\top}$.  Similarly we can show $(A\times_{c} I_{n})\ast X = AX$.  Therefore (i) holds.   
Now we show (ii).  By similar arguments as above, we have for all $(i,j)\in [n]^{2}$ that 
\[ 
\left[(I_{n}\times_{ac} A)\ast X \right]_{ij}=\sum_{k,l} (I_{n}\times_{ac}A)_{ijkl} x_{kl} = \sum_{k,l} \delta_{il} A_{kj} x_{kl}
=\sum_{k} x_{ki}A_{kj} = (X^{\top}A)_{ij}.
\]
Thus $(I_{n}\times_{ac} A)\ast X = X^{\top}A$. Similarly we can show $(A\times_{ac} I_{n}) \ast X = AX$.  Consequently (ii) holds.  
\end{proof}

We call $\A^{c}$ and $\A^{ac}$ the type -I and type-II \emph{Lyapunov transformation} respectively, which are useful in the analysis of the 
stability of systems of common quadratic Lyapunov functions (CQLFs).  A CQLF is defined as 
\beq\label{eq2-3-3} 
V(\bx) = \bx^{\top}P \bx
\eeq 
where $\bx\in\RR^n$ is the state vector and $P\in\RR^{n\times n}$ is a positive definite matrix ($P\succ 0$).  For a linear system 
$\dot{\bx}=A\bx$, this becomes
\beq\label{eq2-3-4} 
\dot{V}(\bx) = \bx^\top (A^\top P + P A) \bx 
\eeq
The system is asymptotically stable if there exists a positive definite matrix $P$ such that $\A^{c}\ast P = AP + PA^{\top}\prec 0$.  For more detail on the Lyapunov-like transformation, please refer to \cite{BGFB1994}.    

Throughout the paper we rely on two basic operations, the contractive product $\A\ast\B$ and the outer product $\A\times\B$ and their extensions 
to arbitrary mode partitions to preserve the multi-index structure of tensors, which classical vectorization inevitably destroys. Our contractive 
product $\A\ast_{[k]}\B$ keeps the result as a tensor of order $p+q-2k$ and the chain rule $\frac{\partl}{\partl T}\bigl(\A(T)\ast_{[k]}\B(T)\bigr)$
is expressed directly in tensor form without $\vecc$ unfolding.  This allows reshape the derivative of a matrix w.r.t. another matrix as a tensor rather 
than as a higher-dimensional vector, and it is the key tool that unifies the ODE and PDE cases in Section~4.  

In standard calculus, the derivative of a matrix-valued function $F(X)$ w.r.t. a matrix $X$ is represented as a second order tensor, i.e., a matrix, 
obtained by unfolding it into a Kronecker matrix.  We show in Section~3 that this derivative can be written directly as a tensor $\frac{\partl F}{\partl X}$ and that the contractions that appear in the chain rule are exactly the 2-contractive products studied in Section~2.  This gives a single
notation that works for ODEs ($\dot{X}=F(X)$), for multi-parameter PDEs ($\partial X/\partial t_{i}=F_{i}(X)$), and for the Lyapunov equations 
discussed below, without switching between $\vecc$ and un-$\vecc$ operators.

We also use our notations to initialze two types of Lyapunov transformations in tensor form, and their action on $X$ are the standard Lyapunov 
operators written in a form compatible with the contractive product.  This compatibility allows us to write the common quadratic Lyapunov 
function $V(\bx)=\bx^{\top}P\bx$ entirely in tensor notation (eq.~\eqref{eq2-3-3}), derive its derivative $\dot{V}$ as a single contraction 
(eq.~\eqref{eq2-3-4}), and carry these expressions through the model-reduction step without ever leaving the tensor algebra.  None of these steps 
is conveniently available in the pure Kronecker-vectorization formalism.

In the next section, we will express some derivatives in tensor forms, which should be useful in the expression of the linear and multilinear ordinary 
differential equations and their solutions.  

\vskip 5pt
\section{Tensor expressions of some derivatives of matrices}
\label{sec3} 
\setcounter{equation}{0} 

Let $\bx=(x_{1},\ldots,x_{m})^{\top}\in\RR^{n}$ and $\by=(y_{1},\ldots,y_{n})^{\top}\in\RR^{n}$ be variable vectors where each component $y_{i}$ is a differentiable function of $\bx$, i.e., $y_{i}=y_{i}(x_{1},\ldots,x_{m})$ for all $i\in [n]$. The derivative $H=\frac{d\by}{d\bx}$ can be defined as an $m\times n$ matrix with
\beq\label{eq:deriv01}
h_{ij} = \frac{\partl y_{j}}{\partl x_{i}}, \forall i\in [m], j\in [n] 
\eeq
Note that $H$ can also be defined as an $n\times m$ matrix with $h_{ij}=\frac{\partl y_{i}}{\partl x_{j}}$ which is the transpose of the matrix 
defined by (\ref{eq:deriv01}). We prefer the definition (\ref{eq:deriv01}) in this paper. Given any constant matrix $A\in\bbC^{m\times n}$ and a 
variable vector $\bx\in\C^{n}$.  Simple computation yields 
\beq\label{eq:deriv02}
\frac{d(A\bx)}{d\bx} = A^{\top}
\eeq
Now we consider the derivative $\frac{dY}{dX}$ where $X\in\bbC^{m\times n}, Y\in\bbC^{p\times q}$ are matrices 
\footnote{Throughout the paper it is supposed that all components of $X$ are supposed to be independent if not otherwise stated.}.  Suppose that each $y_{ij}$ is a differentiable function 
w.r.t. $X$, i.e., $y_{ij}=y_{ij}(x_{11},x_{12},\ldots, x_{mn})$.  The conventional form of the 
derivative $\frac{dY}{dX}$ is defined as 
\beq\label{eq:deriv03}
\frac{dY}{dX} = \frac{d\vecc(Y)}{d\vecc(X)}
\eeq
This traditional definition is not so favorable as it destroys the symmetry.  An alternative expression for the derivative of a matrix is a tensor defined by   
\beq\label{eq:deriv04}
\left(\frac{dY}{dX}\right)_{ijkl} = \frac{\partl y_{kl}}{\partl x_{ij}}
\eeq
where $\frac{dY}{dX}\in\RR^{m\times n\times p\times q}$ is an 4-order tensor. This can be naturally extended to the derivatives of tensors:  let $\X$ 
and $\Y$ be tensors of sizes $m_{1}\times\ldots\times m_{p}$ and $n_{1}\times\ldots\times n_{q}$ respectively.  The derivative $\frac{d\Y}{d\X}$ 
can be defined as a tensor of order $p+q$ whose components are 
\beq\label{eq:deriv05}
\left(\frac{d\Y}{d\X}\right)_{i_{1}\ldots i_{p};j_{1}\ldots j_{q}} = \frac{\partl y_{j_{1}\ldots j_{q}}}{\partl x_{i_{1}\ldots i_{p}}}
\eeq
(\ref{eq:deriv05}) reduces to form (\ref{eq:deriv01}) if $p=q=1$, and it reduces to form (\ref{eq:deriv04}) if $p=q=2$. \par 
\indent Given a square matrix $X$.  Let $X^{\ast}$ be the adjoint matrix of $X$.  Then we have  
\begin{lem}\label{le3-1}
Let $A\in\C^{p\times q}$ be a constant matrix, $X\in\RR^{m\times n}$ a variable matrix, $\la=\la(X)\in\C$ be a function of $X$, and 
$\La=(\la_{ij})\in\C^{m\times n}$ be the matrix with entry $\la_{ij}$ be the derivative of $\la$ w.r.t. $x_{ij}$.  Then 
\begin{description}
\item[(1)]  $\frac{d A}{d X}=0$, i.e., a tensor of size $m\times n\times p\times q$ with all entries being zero.   
\item[(2)]  $\frac{d \la A}{d X}=\La\times A$. 
\item[(3)] $\frac{d(\tr X)}{dX} = I_{n}$ if $m=n$. 
\item[(4)] $\frac{d(\det X)}{dX}= (X^{\ast})^{\top}$ if $m=n$.
\item[(5)] $\frac{d X}{d X}= I_{m}\times_{c} I_{n}$.
\item[(6)] $\frac{d(X^{\top})}{dX}= I_{m}\times_{ac} I_{n}$.
\end{description}
\end{lem}
\begin{proof}
The equalities (1-3) can be verified easily.  Here we just present the proof to (4).  For this purpose, we let $X\in\C^{n\times n}$. Then for any $i\in [n]$
\beq\label{eq3-6} 
\det(X) = \sum_{j=1} x_{ij}(-1)^{i+j}\det X(i|j)  
\eeq
where $X(i|j)$ is the $(n-1)\times (n-1)$ submatrix obtained by the removal of the $i$th row and $j$th column from $X$. Denote 
$X_{ij}=(-1)^{i+j}\det X(i\colon j)$ ($X_{ij}$ is the \emph{algebraic cofactor} w.r.t. $x_{ij}$).  For any given $(i,j)\in [n]\times [n]$, we have 
\[ \frac{d(\det X)}{d x_{ij}} = X_{ij} \]
by (\ref{eq3-6}).  Thus $\frac{d(\det X)}{dX}=(X_{ij})=(X^{\ast})^{\top}$. \par 
\indent To prove (5), we note that both sides of (5) are 4-order tensors of size $m\times n\times m\times n$.  Furthermore, we have by definition
\[  
\left[\frac{d X}{d X}\right]_{ijkl}=\frac{d X_{kl}}{d X_{ij}}=\delta_{ik}\delta_{jl}= (I_{m}\times_{c} I_{n})_{ijkl} 
\]
Thus $\frac{d X}{d X}=I_{m}\times_{c} I_{n}$.  Analoguously we can show (6).  
\end{proof}
   
\begin{thm}\label{th3-2}
Let $A,B$ be constant matrices of appropriate sizes and $Y,Z$ be variable matrices depending on variable matrix $X$. Then we have 
\begin{description}
\item[(a)] $\frac{d (AXB)}{d X}=A^{\top}\times_{c} B$. 
\item[(b)] $\frac{d (YZ)}{d X}=\frac{dY}{dX}\ast_{4} Z +Y\ast_{3}\frac{dZ}{dX}$.
\item[(c)] $\frac{dX^{2}}{d X}=I_{n}\times_{c} X + X^{\top}\times_{c} I_{n}$ if $X\in\RR^{n\times n}$.
\item[(d)]  $\frac{d X^{-1}}{dX} =-X^{-\top}\times_{c} X^{-1}$ if $X$ is invertible.
\end{description}
\end{thm}

\begin{proof}
To prove (a), we may assume that $X\in\RR^{m\times n}$ and $A\in\RR^{p\times m}, B\in\RR^{n\times q}$.  We note that the left hand side (lhs) of (a) is an 4-order tensor with size $m\times n\times p\times q$, which is consistant with that of the right hand side (rhs) of (a). For any 
$(i,j,k,l)\in [m]\times [n]\times [p]\times [q]$, we have 
\beyy 
\left[\frac{d (AXB)}{d X}\right]_{ijkl} 
&=&\frac{d (AXB)_{kl}}{d X_{ij}}\\
&=&\frac{d (\sum_{k^{\p},l^{\p}}A_{kk^{\p}}B_{l^{\p}l}X_{k^{\p}l^{\p}})}{d X_{ij}}\\    
&=&\sum_{k^{\p},l^{\p}}A_{kk^{\p}}B_{l^{\p}l} \frac{d X_{k^{\p}l^{\p}}}{d X_{ij}}\\ 
&=&\sum_{k^{\p},l^{\p}}A_{kk^{\p}}B_{l^{\p}l} \delta_{i k^{\p}}\delta_{j l^{\p}}\\ 
&=&A_{ki}B_{jl} = \left( A^{\top}\times_{c} B\right)_{ijkl}
\eeyy
Thus we have $\frac{d (AXB)}{d X}=A^{\top}\times_{c} B$. \par 
\indent To prove (b), we let $X,Y,Z$ be matrices of size $m\times n, p\times r$ and $r\times q$.  Then the left hand side of (b) is an 4-order tensor of size 
$m\times n\times p\times q$, and it is easy to check that both items in the right hand side of (b) are also 4-order tensors of the same size. Furthermore, 
we have for any $(i,j,k,l)\in [m]\times [n]\times [p]\times [q]$ that  
\beyy 
\left[\frac{d (YZ)}{d X}\right]_{ijkl}&=&\frac{d (YZ)_{kl}}{d X_{ij}}=\frac{d (\sum_{s=1}^{r}Y_{ks}Z_{sl})}{d X_{ij}}\\    
&=&\sum_{s=1}^{r}\left( \frac{d Y_{ks}}{d X_{ij}} Z_{sl} + Y_{ks} \frac{d Z_{sl}}{d X_{ij}}\right) \\ 
&=&\sum_{s=1}^{r}\left[ \left(\frac{d Y}{d X}\right)_{ijks} Z_{sl} + Y_{ks}\left(\frac{d Z}{d X}\right)_{ijsl}\right] \\ 
&=&\left( \frac{dY}{dX}\ast_{4}Z +Y\ast_{3}\frac{dZ}{dX}\right)_{ijsl}
\eeyy 
Thus (b) holds.  Now (c) can be proved by the combination of (b) and (5) of Lemma \ref{le3-1}. In fact, from (b), we have 
 \beyy 
\frac{d (X^{2})}{d X}
& = &X\ast_{3}\frac{dX}{d X} + \frac{dX}{dX}\ast_{4} X \\
& = &X\ast_{3}(I_{n}\times_{c} I_{n}) + (I_{n}\times_{c} I_{n}) \ast_{4} X \\ 
& = &X^{\top}\times_{3} I_{n} + I_{n}\times_{c} X     
\eeyy
The last equality comes from (3) and (4) of  Lemma \ref{le2-2}. \par
\indent To prove (d), we notice that $XX^{-1} = (\det X) I_{n}$.  By (b) and (5) of Lemma \ref{le3-1}, we have 
\beyy
\frac{d(XX^{-1})}{dX} 
&=& \frac{dX}{dX}\ast_{4} X^{-1} + X\ast_{3}\frac{d X^{-1}}{dX}\\
&=& (I_{n}\times_{c} I_{n}) \ast_{4} X^{-1} + X\ast_{3} \frac{d X^{-1}}{dX}\\
&=& I_{n}\times_{c} X^{-1} + X\ast_{3} \frac{d X^{-1}}{dX}
\eeyy
The last equality is due to (4) of Lemma \ref{le2-2}.  Thus we have 
\beq\label{eq:le3-2-1}
X\ast_{3} \frac{d X^{-1}}{dX} = - I_{n}\times_{c} X^{-1} 
\eeq
\end{proof}

Note that if we choose $A=I_{m}$ and $B=I_{n}$ in (a) of Theorem \ref{th3-2},  we can also get (5) of Lemma \ref{le3-1}.

\begin{thm}\label{th3-3}
Let $X\in\RR^{n\times n}$ and $m\ge 1$ be any positive integer.  Then
\beq\label{eq: deriv-Xpower}
\frac{dX^{m}}{d X}=\sum_{s=1}^{m} (X^{s-1})^{\top}\times_{c} X^{m-s}  
\eeq
\end{thm}
\begin{proof}
We use the induction to $m$ to show the result. Note that (\ref{eq: deriv-Xpower}) in the case $m=1$ is immediate from (5) of Lemma 
\ref{le3-1}, and (\ref{eq: deriv-Xpower}) in case $m=2$ can be confirmed by (c) of  Theorem \ref{th3-2}. Now we assume that  
(\ref{eq: deriv-Xpower}) is true for all $k\le m$. We want to show that it holds for $m+1$.  By (b) of Theorem \ref{th3-2}, we have 
\beyy\label{eq:th3-3-1}
\frac{d X^{m+1}}{dX} 
& =& \frac{d X^{m}}{dX}\ast_{4} X + X^{m}\ast_{3} \frac{d X}{dX} \\
& =&\left[\sum_{s=1}^{m} (X^{s-1})^{\top}\times_{c} X^{m-s}\right] \ast_{4} X + X^{m}\ast_{3}(I_{n}\times_{c} I_{n}) \\
& =& \sum_{s=1}^{m}(X^{s-1})^{\top}\times_{c} X^{m+1-s})  + (X^{m})^{\top} \times_{c} I_{n} \\
& =& \sum_{s=1}^{m+1} (X^{s-1})^{\top}\times_{c} X^{m-s} 
\eeyy
The first part of the second last equality is due to the induction hypothesis and the fact that 
\beq\label{eq3-3-2}
(A\times_{c} B)\ast_{4} C = A\times_{c} (BC) 
\eeq
whenever matrix multiplication $BC$ makes sense.  (\ref{eq3-3-2}) can be checked easily by convention. The second part of the second last 
equality is due to (3) of Lemma \ref{le3-1}.  Thus (\ref{eq: deriv-Xpower}) holds.     
\end{proof}

We can deduce easily from Theorem \ref{th3-3} that 
\begin{cor}\label{co3-4}
Let $X\in\RR^{n\times n}$. Then   
\beq\label{eq:th3-3-3}
\frac{dX^{3}}{d X}= I_{n}\times_{c} X^{2} + X^{\top}\times_{c} X + (X^{\top})^{2}\times_{c} I_{n}  
\eeq
\end{cor}

Now we let $X,Y\in\C^{n\times n}$be symmetric and $Y=Y(X)$, each entry $y_{ij}$ is a function of $\set{x_{ij}}$, i.e., 
\[  y_{ij}=y_{ij}(x_{11},\ldots,x_{nn}), \quad  \forall (i,j)\in [n]\times [n]  \]
Denote $N=n(n+1)/2$.  Since $X$ has $N$ indenpendent entries,  each $y_{ij}$ is a $N$-variate function.  We assume that these functions are all 
differentiable.  The traditional form of the derivative $\frac{d Y}{dX}$ is defined as the matrix 
\beq\label{eq3-1:derY-X01}
\frac{d Y}{dX} = \frac{d \vecc_{s}(Y)}{d\vecc_{s}(X)}
\eeq
which is a $N\times N$ matrix. However, the form (\ref{eq3-1:derY-X01}) destroys the symmetry and makes it hard for us to find any latent pattern.  
Now we still use (\ref{eq:deriv04}) to define it and denote $\A=\frac{d Y}{dX}$, then $\A\in\TT_{4;n}$.  By definition and the symmetry of $X$ and $Y$, 
we have 
\[  A_{ijkl} = A_{ijlk} = A_{jikl} = A_{jilk},  \forall (i,j,k,l)\in [n]^{4}. \]
For our purpose, we consider a tensor $\A\in\TT_{2d;n}$.  We call $\A$ a \emph{paired symmetric tensor} if  for any 
$(i_{1},\ldots,i_{d}), (j_{1},\ldots,j_{d})\in [n]^{d}$ and any permutations $\si,\tau\in\Sym_{d}$, we have 
\beq\label{eq3-0:pair-sym-t}
A_{i_1\ldots i_d; j_1\ldots j_d} = A_{i_{\si(1)}\ldots i_{\si(d)}; j_{\tau(1)}\ldots j_{\tau(d)}}
\eeq 
Thus tensor $\A=\frac{d Y}{dX}$ is a paired symmetric tensor in $\TT_{4;n}$.  

To investigate derivatives of symmetric matrices, we define tensor $\X=(X_{ijk})\in\RR^{n\times n\times n}$ associated with $X$ by 
\[
X_{ijk} = \begin{cases}
2X_{ii},  & \texttt{  if\ }  i=j=k,\\
X_{ik},  & \texttt{  if\ }  i=j\neq k \texttt{ or\ } i=k\neq j, \\ 
0,           & \texttt {otherwise}. 
\end{cases}
\]
Then $\X$ is a symmetric tensor w.r.t. modes $\set{2,3}$, that is,  $X_{ijk}=X_{ikj}$ for all $(i,j,k)\in [n]^{3}$.  Now we denote 
\beq\label{eq3-1:Xs} 
\X_{s} := \sum_{\al\subset [4],\abs{\al}=2} I_{n}\times_{\al} X
\eeq
where the summation runs over all 2-sets $\al$ of $[4]$. So $\al$ can be any one of  
\[ \set{1,2},\quad \set{1,3},\quad \set{1,4},\quad \set{2,3},\quad \set{2,4},\quad \set{3,4}. \]  
If we denote 
\beq\label{eq3-1:X-I-n} 
\X^{nat} := I_{n}\times_{(1,2)} X + X\times_{(1,2)} I_{n}
\eeq
Then we have 
\beq\label{eq3-3-4: c+ac=s}
\X^{c}+\X^{ac} = \X_{s} - \X^{nat}
\eeq
where $\X^{c}$ and $\X^{ac}$ are defined respectively by (\ref{eq2-3-1}) and (\ref{eq2-3-2}).  We can show that 
\begin{lem}\label{le3-5}
Let $X\in\RR^{n\times n}$ be symmetric. Then we have 
\begin{description}
\item[(1)]  $\X_{s}$ is a symmetric tensor.  
\item[(2)]  $\X^{c}, \X^{ac}$ and $\X^{nat}$ are paired symmetric tensors. 
\end{description} 
\end{lem}
\begin{proof}
To prove (1), we note that (\ref{eq3-1:Xs}) is equivalent to  
\beq\label{eq3-5-1}
X^{s}_{ijkl}=\delta_{ij}x_{kl}+\delta_{ik}x_{jl} +\delta_{il}x_{jk} +\delta_{jk}x_{il} + \delta_{jl}x_{ik} +\delta_{kl}x_{ij}  
\eeq
for each $\al:=(i,j,k,l)\in [n]^{4}$.  To show the symmetry of $\X^{s}$, it suffices to show that  each $X^{s}_{ijkl}$ is invariant under any transposition of its index.  For example
\beq\label{eq3-5-2}
X^{s}_{jikl}=\delta_{ji}x_{kl}+\delta_{jk}x_{il} +\delta_{jl}x_{ik} +\delta_{ik}x_{jl} + \delta_{il}x_{jk} +\delta_{kl}x_{ji}  
\eeq
We can see that $X^{s}_{ijkl}=X^{s}_{jikl}$ by comparing the right side of (\ref{eq3-5-1}) and (\ref{eq3-5-2}) while noticing the symmetry of $X$
($x_{ij}=x_{ji}$) and $I_{n}$ (i.e., $\delta_{ij}=\delta_{ji}$).  The invariance of the components under other index transpositions can also be verified 
analogously. \par 
\indent  To prove (2), we notice by definition that 
\beq\label{eq3-3-5} 
X^{c}_{\si}= (I_{n}\times_{c} X+X\times_{c} I_{n})_{ijkl} = \delta_{ik}x_{jl}+x_{ik}\delta_{jl} 
\eeq
Then we can check directly that $\X^{c}$ is paired symmetric.  Similarly we can prove the symmetry of  $\X^{ac}$.  For $\X^{nat}$, we note 
its components satisfy 
\[
 X^{nat}_{ijkl} = (I_{n}\times X+X\times I_{n})_{ijkl} = \delta_{ij}x_{kl} + x_{ij}\delta_{kl}, \]
 which follows that 
\[
X^{nat}_{ijkl} = \begin{cases}
x_{ii} +x_{kk},  & \texttt{  if\ }  i=j, k=l,\\
x_{kl},  & \texttt{  if\ }  i=j, k\neq l,\\
x_{ij},   & \texttt{ if\ }  i\neq j, k=l,\\ 
0,           & \texttt {otherwise}. 
\end{cases}
\]
By definition, we can see that $\X^{nat}$ is pair symmetric.
\end{proof}

Now we can state our main results on the derivative of symmetric tensors. 
\begin{thm}\label{th3-5}
Let $X=(x_{ij})\in\RR^{n\times n}$ be a symmetric matrix. Then we have 
\begin{description}
\item[(1)] $\frac{dX}{dX} = I_{n}\times_{c} I_{n} +I_{n}\times_{ac} I_{n} -\caI_{4,n}$. 
\item[(2)] $\frac{d X^{2}}{dX} = \X_{s} - \X^{nat} - I_{n}\times \X $.  
\end{description}
\end{thm}

\begin{proof}
To prove (1), we denote $\A=\frac{dX}{dX}=(A_{ijkl})$ and let $\B$ be the tensor of the rhs of (1).  Note that by definition and the symmetry of $X$, we 
have for all $(i,j,k,l)\in [n]^{4}$ that $A_{ijkl}=1$  if and only if  $i=k, j=l$ or $i=l, j=k$, i.e., all the nonzero components of $\A$ are listed as follows:
\begin{description}
\item[(i)]  $A_{iiii} =1, i\in [n]$.
\item[(ii)]  $A_{ijij} =1, (i,j)\in [n]\times [n]$.
\item[(iii)]  $A_{ijji} =1, (i,j)\in [n]\times [n]$.
\end{description}
All other components of $\A$ are zeros.  We also have 
\beyy 
B_{ijkl} 
&=& (I_{n}\times_{c} I_{n})_{ijkl} + (I_{n}\times_{ac} I_{n})_{ijkl} - (\caI_{4,n})_{ijkl}\\
&=& \delta_{ik}\delta_{jl} + \delta_{il}\delta_{jk}  - \delta_{ijkl}
\eeyy
It follows that, for all distinct $i,j\in [n]$, we have   
\[ B_{ijij}=\delta_{ii}\delta_{jj} + \delta_{ij}\delta_{ji}  - \delta_{ijij}=1+0-0 = 1 =A_{ijij}, \]
and 
\[ B_{ijji}=\delta_{ij}\delta_{ji} + \delta_{ii}\delta_{jj}  - \delta_{ij}\delta_{ji} =0+1-0=1=A_{ijji}. \]
For $i=j=k=l \in [n]$, we have 
\[ B_{iiii} = \delta_{ii}^{2} + \delta_{ii}^{2} - \delta_{iiii}=1+1-1 = 1 = A_{iiii},  \]
and all other components of $\B$ are zeros. So $\B$ is identical to $\A$. Thus (1) is proved. \par 
\indent To prove (2). We use (b) of Theorem \ref{th3-2} and (1) to get 
\beyy
\frac{d X^{2}}{dX} 
&=&\left( \frac{d X}{dX}\right)\ast_{4} X + X\ast_{3} \frac{dX}{dX}\\
&=&\left(I_{n}\times_{c} I_{n}+I_{n}\times_{ac} I_{n}-\caI_{4;n}\right)\ast_{4} X + X\ast_{3}\left(I_{n}\times_{c} I_{n}+I_{n}\times_{ac} I_{n}-\caI_{4,n}\right)\\     
&=& I_{n}\times_{c} X +X\times_{ac} I_{n} +X\times_{c} I_{n}+X\times_{c} I_{n} +I_{n}\times_{ac} X-(\caI_{4;n}\ast_{4} X+X\ast_{3}\caI_{4;n})\\
&=& I_{n}\times_{c} X +X\times_{ac} I_{n} +X\times_{c} I_{n}+X\times_{c} I_{n} +I_{n}\times_{ac} X -I_{n}\times \X\\
&=& \X_{s} -( I_{n}\times X+X\times I_{n} + I_{n}\times \X)\\ 
&=& \X_{s} - \X^{nat} -  I_{n}\times \X 
\eeyy  
Thus item (2) is proved. 
\end{proof}
 
Given an even-order tensor $\A\in\TT_{2d;n}$ and a positive integer $k$. The power $\A^{k}$ can be defined similar to (\ref{eq:tpower4}), i.e.,  
\beq\label{eq:Tpower-general} 
\A^{0}:=\caI_{2d;n}, \quad \A^{1}:=\A, \quad \A^{k+1}:=\A\ast\A^{k}, \forall k=2,3,\ldots . 
\eeq
where the product $\A\ast\B:= \A\ast_{[d]} \B$ is defined on $\TT_{2d;n}$ by (\ref{eq:t-t-contract}).  Recall that the identity tensor 
$\caI:=\caI_{2d;n}=\dsb{I_{n},\ldots,I_{n}}[\pi]$ with partition $\pi=\set{\set{1,d},\set{2,d+1},\ldots, \set{d,2d}}$. Similar to (\ref{th3-5}), 
we have
\begin{thm}\label{th3-6} 
Let $\X\in\TT_{d;n}$ be an $d$-order tensor and $k$ be any positive integer.  Then we have 
\beq\label{eq:der4t-t}
\frac{d\X}{d\X} = \caI_{2d;n}
\eeq
\end{thm}
\begin{proof}
Denote $\A=\frac{d\X}{d\X}$. Then $\A\in\TT_{2d;n}$ whose entry indexed by $(i_{1},\ldots, i_{d}, j_{1},\ldots,j_{d})$ is 
\beyy 
A_{i_{1}\ldots i_{d};j_{1}\ldots j_{d}} 
&=& \frac{dX_{j_{1}\ldots j_{d}}}{dX_{i_{1}\ldots i_{d}}} \\
&=& \delta_{i_{1}j_{1}}\ldots \delta_{i_{d}j_{d}} \\ 
&=& \left[ I_{n},\ldots, I_{n}][\pi]\right]_{i_{1}\ldots i_{d};j_{1}\ldots j_{d}} 
\eeyy
where $\caI_{2d;n}= [I_{n},\ldots, I_{n}][\pi]$ with $\pi=\set{\set{1,d},\set{2,d+1},\ldots, \set{d, 2d}}$.  Thus (\ref{eq:der4t-t}) is proved. 
\end{proof}

A more general form for derivates of the power $\X^{k}$ w.r.t. $\X$ is much more complicate than the case when $X$ is a matrix, and we 
will not discuss it here.  In the next section, we will express the linear differential equations and present its solution in tensor forms. 

\vskip 5pt
\section{Tensor expressions of Linear Ordinary Differential Equations and their solutions}
\label{sec4} 
\setcounter{equation}{0} 

By tensor, we can simplify some ordinary differential equations or some partial differential equations. We start with a simple linear form of ODE 
\beq\label{eq4-1:3ode01}
x^{(3)} = a_{1}x +a_{2}x^{\p} +a_{3} x^{\p\p}+ f(x,x^{\p},x^{\p\p}) 
\eeq
where $x^{(k)}$ denotes the $k$th derivative of $x$ (w.r.t. $t$) for $k\ge 3$, $x^{\p}$ and $x^{\p\p}$ denote respectively the first and second 
derivative of $x$ w.r.t. $t$, and $f(x,y,z):=\bx^{\top}A\bx$ is a quadratic form determined by symmetric matrix $A=(a_{ij})\in\RR^{3\times 3}$ 
where $\bx=(x,y,z)^{\top}\in\RR^{3}$.  
\begin{lem}\label{le4-1}
The ODE defined by (\ref{eq4-1:3ode01}) can be expressed as the tensor form 
\beq\label{eq4-2:3ode03} 
\frac{d \bx}{dt} = \al^{\top}\bx+\A\bx^2 
\eeq 
where $\al=(a_{1},a_{2},a_{3})^{\top}, \bx\in\RR^3$, and $\A$ is an $3\times 3\times 3$ tensor. 
\end{lem}
\begin{proof}
We denote  
\[ 
\bx = (x_{1},x_{2},x_{3})^{\p}, \quad  x_{1}=x,  x_{2}=x^{\p}, x_{3}=x^{\p\p}.
\]
Then (\ref{eq4-1:3ode01}) is equivalent to   
\bey\label{eq4-3:3ode04}
\begin{aligned}
	x_{1}^{\p} &= x_{2} \\
	x_{2}^{\p} &= x_{3} \\
	x_{3}^{\p} &= \al^{\top}\bx +f(\bx) 
\end{aligned}
\eey
which can be equivalently written as 
\beq\label{eq4-4:3ode05}
\frac{d\bx}{dt} = B\bx +\hat{f}
\eeq
where $B\in\RR^{3\times 3}$ whose first and second row are respectively the second and third row of the identity matrix $I_{3}$ and the third
row is $\al^{\top}$, and $\hat{f}=(0,0,f)^{\top}$.  Now we define an $3\times 3\times 3$ tensor $\A=(A_{ijk})$ such that  
\[ A(1, :, :) =0, A(2,:,:) =0, A(3,:,:) =A. \]
It suffices to show that $\hat{f}=\A\bx^2$. In fact, we have for $i=1,2$
\[ (\A\bx^2)_i =\sum_{j,k=1}^{3} A(i,j,k)x_j x_k =0 = \hat{f}_i, \]
and for $i=3$,
\[ (\A\bx^2)_3 =\sum_{j,k=1}^{3} A(3,j,k)x_j x_k =\bx^{\top}A\bx=f(\bx)=\hat{f}_3. \]
Thus (\ref{eq4-2:3ode03}) is proved.
\end{proof}

Now consider an $n$-order linear homogeneous ODE  
\beq\label{eq4-5}
  x^{(n)}+ a_{n-1}x^{(n-1)}+\ldots+ a_1x^{\p} +a_0x=0, \quad  a_{j}\in \bbC,
\eeq
(\ref{eq4-5}) can be reformulated as 
\beq\label{eq4-6}
\frac{d\by}{dt} = A\by  
\eeq 
where $y_{1}=x, y_{2}=\dot{x}, y_{3}=\ddot{x}, \ldots, y_{k}=x^{(k-1)}$ and $\by=(y_{1},y_{2},\ldots, y_{n})^{\top}$, and 
\beq\label{eq4-7}
A = \begin{bmatrix} 
0 & 1&0&0&\cdots & 0\\  
0 & 0&1&0&\cdots & 0\\  
\vdots & \vdots &\vdots &\vdots &\cdots & \vdots \\
0 & 0&0&0&\cdots & 1\\
-a_{0} & -a_{1}&-a_{2}&-a_{3}&\cdots & -a_{n-1}\end{bmatrix} 
\eeq 
Denote $f(x):=\sum_{k=0}^{n} a_k x^k$ with $a_n=1$.  $A$ is called the \emph{companion matrix} of  $f(x)$. Now let 
$\bx(t)=(x_{1},x_{2},\ldots,x_{p})^{\top}$ with each $x_i=x_{i}(t)$ sufficiently differentiable for $i\in [p]$.  We denote 
$\bx^{(k)}=(x_{1}^{(k)},\ldots, x_{p}^{(k)})^{\top}$ where $x_{i}^{(k)}$ is the $k$th derivative of $x_i$, and let 
$A_{k}\in\RR^{p\times p}$ be a constant matrix for each $k\in [n]$.  The extension of (\ref{eq4-5}) is in form 
\beq\label{eq4-8}
\bx^{(n)} + A_{n-1}\bx^{(n-1)}+\ldots +A_{1}\bx^{(1)} + A_{0}\bx = 0
\eeq
Denote by $S(n)$ the set of all linearly independent solutions of (\ref{eq4-6}). We have 
\begin{lem}\label{le4-2:unidiff}
Let $f(x) = \sum\limits_{k=0}^{n} a_{k}x^{k}\in R[x]$ be a polynomial of degree $n$ with $a_{n}=1$, and 
$\set{\la_{1},\la_{2},\ldots,\la_{r}}$ be the set of all distinct roots of $f$ with $m_{i}$ the multiplicity of $\la_{i}$ for $i\in [r]$.  
Then we have 
\beq\label{eq4-9}  
S(n) =\set{ \sum_{i,j} t^{j-1}e^{\la_{i}t}\colon  \forall j\in [m_{i}], i\in [r] }  
\eeq   
\end{lem}
It is obvious that $\abs{S(n)}=n$ since $n=m_{1}+m_{2}+\ldots + m_{r}$.  The proof of Lemma \ref{le4-2:unidiff} can be found in \cite{CC1997}.\\

Now consider matrix sequence $A_0,A_1, \ldots,A_m\in\RR^{p\times q}$. Denote $A_{ij}^{(k)}$ for the $(i,j)$-entry of $A_k$ and let 
$\bx=(x_1,\ldots, x_p)^{\top}$ with each $x_i$ being sufficiently differentiable w.r.t. $t$.  Let $X\in\RR^{p\times n}$ with entries 
$X_{ij}=x_{i}^{(j-1)}$ and $x_{i}^{(0)}:= x_{i}$ for $i\in [p], j\in [n]$.    
\begin{definition}\label{def4-3}
Given ODE (\ref{eq4-8}) with each $A_{i}\in\RR^{p\times p}$ being constant.  The \emph{coefficient tensor} $\A$, or called \emph{C-tensor}, 
is a tensor of size $p\times n\times p\times n$, whose components are defined by   
 \beq\label{eq4-10} 
 A_{ijkl}= \begin{cases}  
   1 & \text{if }\  k=i, l=j+1, 1\le j\le n-1, \\
   -A^{(l-1)}_{ik} & \text{if }\  j=n, \\    
   0 & \text{otherwise}. 
 \end{cases}  
 \eeq
 where $A^{(s)}_{ij}$ is the $(i,j)$th component of $A_{s}$ for all $s\in [n]-1$\footnote{ $[n]-1:=\set{0,1,2,\ldots,n-1 }$}. 
\end{definition}

Now we are ready to state  
\begin{thm}\label{th4-2}
The multivariate ODE (\ref{eq4-8}) with coefficient matrices $A_{0},A_{1},\ldots, A_{n-1}$ can be reformulated as 
\beq\label{eq4-11}  
\frac{d X}{dt} = \A\ast X  
\eeq
where $\A$ is the C-tensor of size $p\times n\times p\times n$ defined by \eqref{eq4-10}, $X\in\RR^{p\times n}$ is the matrix with 
$X_{ij} = x_{i}^{(j-1)}$, and $\frac{dX}{dt}\in\RR^{p\times n}$ satisfies $(\frac{dX}{dt})_{ij}=\frac{d X_{ij}}{dt}$. 
\end{thm}

\begin{proof}
We need to show that $(\frac{dX}{dt})_{ij}=(\A X)_{ij}$ for all $i\in [p], j\in [n]$.  Note that 
$X_{ij} =x_{i}^{(j-1)}$.  We consider two cases: \par 
\indent (1).  $j\in [n-1]$. By definition, we have 
\beyy  
(\A\ast X)_{ij} &=& \sum\limits_{k,l} A_{ijkl}X_{kl} \\
                    &=& A_{iji(j+1)} X_{i,j+1} =X_{i,j+1}\\
                    &=& x_{i}^{(j)} =\frac{d X_{ij}}{dt}  
\eeyy  
Thus $\frac{d X_{ij}}{dt}= (\A X)_{ij} $ for all $i\in [p], j\in [n-1]$. \\
\indent (2).  $j=n$.  In this situation, we have 
\beyy  
(\A\ast X)_{i n} &=& \sum\limits_{k,l} A_{i n k l}X_{k l} \\
             &=& -\sum\limits_{k,l} A_{ik}^{(l-1)} x_{k}^{(l-1)}\\
             &=& -\bigl[\sum_{l=1}^{n} \bigl(A^{(l-1)}x^{(l-1)}\bigr]_{i}\\
             &=& x_{i}^{(n)}  
\eeyy  
Note that 
\[
x_{i}^{n} = \frac{d\!x_{i}^{(n-1)}}{d\!t} = \frac{d\!X_{i,n}}{d\!t}. \]
It follows that  $(\A\ast X)_{i,n} = \frac{d\!X_{i,n}}{d\!t} = \bigl(\frac{d\!X}{d\!t}\bigr)_{i,n}$.  Thus \eqref{eq4-11} holds.   
The proof is completed.   
\end{proof}

Theorem \ref{th4-2} can also be proved by the combination of the vectorization and matricisation of tensors. Recall that the vectorization 
of a matrix $X\in\RR^{m\times n}$ maps $X$ to a vector $\bx\in\RR^{mn}$ by stacking all columns of $X$ in order, and the balanced matricisation 
of the C-tensor $\A$ defined by \eqref{eq4-10} yields an $pn\times pn$ matrix in form 
\beq\label{eq4-12}
A = \begin{pmatrix} 
         0 & I_{p}  &        0 & \cdots & & 0\\
         0 &        0  & I_{p} & \cdots & & 0\\
 \vdots & \vdots & \ddots & \cdots & &   \\
            &            &           &            &  &I_{p}\\
  -A_{0}&-A_{1}&\cdots &\cdots&   &-A_{n-1}  
\end{pmatrix}
\eeq
Denote by $\by=\vecc(X)$ where $X=[\bx,\bx^{(1)},\ldots, \bx^{(n-1)}], \bx^{(k)}=\frac{d^{k}\bx}{dt^{k}}$. Then 
\[
\by=\begin{pmatrix} \bx\\  \bx^{(1)}\\ \vdots \\ \bx^{(n-1)}\end{pmatrix},
\]
Then 
\[
 \vecc(\A\ast X) = A\vecc(X) =A\by = 
 \begin{pmatrix} \bx^{(1)}\\  \bx^{(2)}\\ \vdots \\ \bx^{(n-1)}\\ -\sum_{k=0}^{n-1} A_{k}\bx^{(k)}\end{pmatrix}
\]
Thus (\ref{eq4-11}) is equivalent to (\ref{eq4-8}). \\

For $n=1$ we have $\A=-A_{0}\in\RR^{p\times p}$.  For $n=2$, $\A$ is a tensor of size $p\times 2\times p\times 2$ defined by 
\[ A(:,1,:,1) =0\in\RR^{p\times p},  A(:,1,:,2) =I_{p}, A(:,2,:,1) =-A_{0},  A(:,2,:,2) =-A_{1}. \]    

The tensor-matrix form (\ref{eq4-11}) generalizes the first-order state-space representation to higher-order tensor dynamics, and unifies the analysis of multivariate ODEs, enabling insights into the stability, control, and computational efficiency. It makes possible for us to extend this framework to nonlinear tensor ODEs or quantum tensor networks in our future work. In the next section, we will also unify the partial differential equations into the tensor-matrix form. \\
 

Now consider the solution to the ODE (\ref{eq4-11}) where $\A$ is a constant tensor of size $p\times n\times p\times n$, $X\in\RR^{p\times n}$ is 
a matrix with $X_{ij}$ depending on $t$, and all the entries of $X$ are mutually independent.  We have 
\begin{thm}\label{th4-3}
The general solution to the ODE (\ref{eq4-11}) with initial condition $C=X(0)\in\RR^{p\times n}$ is 
\beq\label{eq4-13}
X = \exp(t\A)\ast C  
\eeq
\end{thm}
\begin{proof}
Denote $Y= \exp(t\A)\ast C$.  By Lemma \ref{le2-4} we have 
\[
\frac{d Y}{dt}= \left(\frac{d}{dt} \exp(t\A)\right)\ast C =\A\ast \left(\exp(t\A)\ast C\right) =\A\ast Y   
\]
It follows that (\ref{eq4-13}) is the solution to (\ref{eq4-11})  with initial condition $X(0)=C$.  
\end{proof}

For $n=1$, (\ref{eq4-13}) reduces to $\bx=e^{-tA_{0}}c$, which is the solution to $\frac{d\bx}{dt}=A\bx$ with initial condition 
$\bx(0)=c\in\RR^{p}$. In fact, tensor $\A\in\RR^{p\times 1\times p\times 1}$ in this case is identical to matrix $-A_{0}$.  

%
 
A C-tensor $\A$ of size $p\times n\times p\times n$ defined by Definition \ref{def4-3} can be flattened in a balance-way (i.e., mode-$\set{1,2}$ 
to row and mode-$\set{3,4}$ to column ) as a matrix $\tilde{A}$ of size $pn\times pn$.  A necessary and sufficient condition for the system to be 
asymptotic stable is that $\Re(\la)<0$ for all eigenvalues $\la$ of $\tilde{A}$. 

The tensor-matrix form \eqref{eq4-11} can be used to simplify the traditional multivariate ODE form (\ref{eq4-5}), but also help us understand 
the solution to (\ref{eq4-5}). For a special case when $n=p$,  we find that we can solve Equation (\ref{eq4-11}) easily by employing the powers 
of tensors that built upon the multiplication we just defined before. Note that the $\A^{2}$ can be understood as the product $\A\cdot \A$ which 
turns out to be the same size as $\A$ by definition.

Now suppose that $\X\in\TT_{q}$, each of whose components is a function of $\T\in\TT_{p}$, where the sizes of $\X$ and $\T$ are respectively 
$\bfn:=n_{1}\times \ldots \times n_{q}$ and $\bfm:=m_{1}\times \ldots \times m_{p}$.  The product $\bfm\times \bfn$, interpreted as 
\[ 
m_{1}\times \ldots \times m_{p}\times n_{1}\times \ldots \times n_{q},  
\]
serves as the size of a tensor of order $p+q$.  Then (\ref{eq4-11}) can be extended to a more general tensor form 
\beq\label{eq4-14}
\frac{d\X}{d\T} = \A\ast \X
\eeq
where $\A$ is a constant $(p+2q)$-order tensor of size $\bfm\times \bfn\times \bfn$, and $\frac{dX}{dT}$ is an $(p+q)$-order tensor with size 
$\bfm\times\bfn$.  Here $\A\ast \X=\A\ast_{[q]}\X$.  Thus (\ref{eq4-14}) is equivalent to 
\beq\label{eq4-15}
\left(\frac{d\X}{d\T}\right)_{i_{1}\ldots i_{p} j_{1}\ldots j_{q}} 
= \sum_{k_{1},\ldots,k_{q}} A_{i_{1}\ldots i_{p} j_{1}\ldots j_{q}k_{1}\ldots k_{q}}X_{k_{1}\ldots k_{q}}
\eeq
For $p=q=1$, both sides of (\ref{eq4-14}) are 2-order tensors (matrices). If $\bx\in\RR^{n},\bft\in\RR^{m}$, (\ref{eq4-14}) becomes 
\beq\label{eq4-16}
\frac{d\bx}{d\bft} = \A\ast\bx
\eeq
where $\frac{d\bx}{d\bft}\in\RR^{m\times n}, \A\in\RR^{m\times n\times n}$.  Equation (\ref{eq4-16}) can be interpreted as the $m$ 
linear partial differential equations 
\beq\label{eq4-17}
\frac{\partl \bx}{\partl t_{i}} =A_{i}\bx, \quad  i=1,2,\ldots, m.   
\eeq
where $A_{i}=\A(i,:,:)\in\RR^{n\times n}$.   

In order to state the following result, we define the \emph{commutator} of two matrices $A$ and $B$ as $[A, B] = AB -BA$ where $A$ and $B$ are both $n\times n$ matrices.  Hence $A$ and $B$ are called \emph{commutative} if $[A,B] = \mathbf{0}$.  We have    

\begin{thm}\label{th4-4}
Let $\A\in\RR^{m\times n\times n}$ be the third-order tensor with $\A(i,:,:) = A_{i}$ where $A_{1},\dots, A_{m}\in\RR^{n\times n}$ be a sequence 
of constant matrices satisfying condition 
\beq\label{eq4-17:commut}
[A_{i},A_{j}] = \mathbf{0}, \qquad   \forall  i,j\in [m]
\eeq
Then the solution to (\ref{eq4-16}) under the initial condition $\bx(0)=\bfc\in\RR^{n}$ with $\bx\in\RR^{n}, \bft\in\RR^{m}$ is 
\beq\label{eq4-18}
\bx = \exp (\A\ast_{1} \bft)\ast \bfc 
\eeq
\end{thm}
\begin{proof}
We denote by $A[k]$ the flattened matrix of $\A$ along mode $k$ for $k\in \set{1,2,3}$, and $M_{j}$ the $j$th column vector of a matrix $M$.
Then we have 
\beq\label{eq4-19-1} 
\A\ast_{k}\bx=\sum_{j} x_{j}A[k]_{j},  \forall k=1,2,3.  
\eeq
We first show that for every positive integer $q$
\beq\label{eq4-20}
\frac{d (\A\ast_{1} \bft)^{q}}{d \bft} = q \A\ast (\A\ast_{1} \bft)^{q-1} 
\eeq
Note that $(\A\ast_{1} \bft)^{q} = (\sum_{i=1}^{m }t_{i}A_{i})^{q}$ where $A_{i}=\A(i,:,:)$.  Since \eqref{eq4-17:commut} holds, we have 
\beq\label{eq4-21}
\frac{\partl (\A\ast_{1}\bft)^{q}}{\partl t_{i}} = q A_{i} (\A\ast_{1}\bft)^{q-1}. 
\eeq
Thus 
\beq\label{eq4-22}  
\frac{d (\A\ast_{1}\bft)^{q}}{d \bft} = q \A\ast (\A\ast_{s}\bft)^{q-1}, 
\eeq 
By (\ref{eq4-22}) and the Taylor expansion of $\exp(\A\ast \bft)$, we get 
\beq\label{eq4-23}  
\frac{d (\exp\set{\A\ast_{1}\bft})}{d \bft} = n \A\ast_{s} \exp\set{\A\ast_{1}\bft} 
\eeq 
Thus we can deduce that (\ref{eq4-18}) satisfies equation (\ref{eq4-16}) with the initial condition. 
\end{proof}

Equation (\ref{eq4-14}) offers a more concise representation and enables a substantial reduction in computational cost.  Note that in 
Theorem \ref{th4-4} the general solution to \eqref{eq4-16} may not be in form \eqref{eq4-18} if \eqref{eq4-17:commut} does not hold.  The following example illustrates this issue. 

\begin{exm}\label{exm4-1}

We define a third-order tensor $\A\in\RR^{2\times2\times2}$ by its horizontal slices 
\[
A_1 =\A(1,:,:)= \begin{pmatrix}1&0\\0&0\end{pmatrix},\qquad
A_2 =\A(2,:,:)= \begin{pmatrix}0&1\\0&0\end{pmatrix}.
\]
It is easy to check that $A_{1}A_{2}\neq A_{2}A_{1}$.  Now we let $\bx=(x_1,x_2)^\top, \bft =(t_{1},t_{2})^{\top}$ and consider the solution to \eqref{eq4-16} under initial condition $\bx(0) = \bfc =(c_{1},c_{2})^{\top}$.  Then \eqref{eq4-16} in this case is equivalent to the ODE system 
\beq\label{eq:ex4-7:ode}
\frac{\partl x_{1}}{\partl t_{1}}=x_{1}, \frac{\partl x_{2}}{\partl t_{1}} = x_{2}, 
\frac{\partl x_{1}}{\partl t_{2}}=0,  \frac{\partl x_{2}}{\partl t_{2}} = 0. 
\eeq 
A simple calculation yields its solution 
\beq\label{eq:ex4-7:sol}
 \bx = \begin{pmatrix}x_1 & x_2\\0&0\end{pmatrix}\bfc =  \begin{pmatrix} c_1e^{t_1} \\ c_2e^{t_2} \end{pmatrix}
\eeq
Thus we have the solution $x_{1} = c_1e^{t_1}, x_{2} = c_2e^{t_2}$.  On the other hand, \eqref{eq4-18} produces a solution vector 
\[
\by:=\exp\set{\A\ast\bft}\bfc = \exp\set{\begin{pmatrix}x_1 & x_2\\0&0\end{pmatrix}}\bfc 
= \frac{1}{t_{1}}\begin{pmatrix} (c_1t_1+c_2t_2)e^{t_1}- c_2 t_2 \\ c_2 t_1 \end{pmatrix}
\]
which yields 
\beyy 
x_{1} &=& \frac{1}{t_{1}} \bigl[(c_1t_1+c_2t_2)e^{t_1}- c_2 t_2\bigr], \\
x_{2} &=& c_{2},
\eeyy
a conflict with the correct answer $x_{2} = c_2e^{t_2}$.  
\end{exm}

\vskip 5pt
\section{Partial Tucker Decomposition for Higher Order Tensors and Efficient ODE Solution}
\label{sec5}
\setcounter{equation}{0}

\subsection{Definition of parTuckerD}
\label{sec5-1}

Given an $m$th-order tensor $\A\in\TT_{m}$ of size $n_1\times\cdots\times n_m$, the full Tucker decomposition compresses \emph{all} modes simultaneously:
\[
\A\approx \caG\ast_1 U^{(1)} \ast_2 U^{(2)} \cdots \ast_m U^{(m)},
\]
with factor matrices $U^{(d)} \in \RR^{r_d \times n_d}$ where $r_{d}\ll n_{d}$ for each $d\in [m]$.  Here the mode product $\A\ast_{k} U$ is defined as
\[
\bigl\{\A\ast_{k} U\bigr\}_{i_1i_2\dots i_m}
= \sum_{i_k'=1}^{n_{k}} A_{i_1\dots i_{k-1}i_k'i_{k+1}\cdots i_m} U_{i_k' i_k}, \quad \forall (i_1, i_2, \ldots, i_m),
\]
which differs from the conventional definition (see e.g.\ \cite{Kolda2009}).

In many real-world applications, certain modes carry intrinsic structural or semantic meaning that should \emph{not} be compressed, while others are highly redundant and benefit greatly from dimension reduction.  The \emph{partial Tucker decomposition} (\emph{parTuckerD}) addresses this by introducing a \emph{compressed mode set}.

\begin{definition}[Partial Tucker Decomposition]
Let $\caM\subseteq[m]$ be a nonempty subset of modes to be compressed, and let $\caM^c=[m]\setminus\caM$ be the uncompressed modes.  A partial Tucker decomposition of $\A$ with respect to $\caM$ seeks factor matrices $U^{(d)}\in\RR^{r_d\times n_d}$ for $d\in\caM$ and a core tensor $\caG\in\TT_{m}$ of size $s_1\times\cdots\times s_m$ (where $s_d=r_d$ for $d\in\caM$ and $s_d=n_d$ for $d\in\caM^c$) such that
\beq\label{eq:parTuckerD_m}
\A \approx \caG\dsb{U^{(d)}}_{\caM}
:=\caG\ast_{d_1} U^{(d_1)}\ast_{d_2} U^{(d_2)}\ast\cdots\ast_{d_p} U^{(d_p)},
\eeq
where $\caM=\{d_1,\ldots,d_p\}$.  Each factor matrix satisfies
\beq\label{eq:Ud}
U^{(d)}U^{(d)\top} = I_{r_d}, \qquad r_d\ll n_d,\;\forall d\in\caM.
\eeq
Here $U^{(d)\top} = \left(U^{(d)}\right)^{\top}$.  If each $r_{d_k}$ is the smallest possible, the $p$-tuple $(r_{d_1},\ldots,r_{d_p})$ is called the \emph{Tucker rank} of the parTuckerD \eqref{eq:parTuckerD_m}, and $\caM$ is called the \emph{compressed mode set} (CM set).
\end{definition}

Denote $V_{d}:=\left(U^{(d)}\right)^{\top}$. Then $V_{d}\in\RR^{n_d\times r_d}$ (tall-and-skinny), so the mode-$d$ contraction 
$\A\ast_d V_{d}$ reduces mode $d$ from dimension $n_d$ to $r_d$, which is consistent with the mode-product definition above and with 
the numerical examples in Section~\ref{sec6}.

Given $\caM=\{d_1,\ldots,d_p\}$, let
\[
\U=\dsb{U^{(d_1)}, U^{(d_2)}, \ldots, U^{(d_p)}}_{\pi}\in\TT_{2p},
\]
where $\pi=\{\{1,p+1\},\{2,p+2\},\ldots,\{p,2p\}\}$ partitions $[2p]$.  By Corollary~\ref{co2-3-1: TD}, the parTuckerD approximation can be written compactly as
\beq\label{eq:parTuckerD2}
\A \approx \caG\dsb{U^{(d)}}_{\caM} = \caG\ast_{\caM}\U,
\eeq
that is, $\A$ can be expressed as a contractive product of the core tensor $\caG\in\TT_{m}$ and the factor tensor $\U\in\TT_{2p}$.  

\subsection{Choosing the Compressed Mode Set \texorpdfstring{$\caM$}{M}}
\label{sec5-2}

The flexibility of parTuckerD is its main advantage over full Tucker decomposition: each mode in $\caM$ can have its own target rank $r_d$, adapting to the actual low-rank structure, and forcing low rank on inherently high-rank modes (which can amplify noise) is avoided.

The choice of $\caM$ depends on the semantic meaning of each mode and the dimensionality pattern. Basically, if $n_d$ is very large for some 
$d\in [m]$ or mode $d$ exhibits approximate low-rank structure (e.g., spatial or temporal modes with smooth variation), then we let $d\in\caM$.
On the other hand, we may let $d\in\caM^{c}$ if mode $d$ carries essential categorical or ordinal information that must be preserved exactly 
(e.g., patient IDs, time indices, subject labels).  In practice, $\caM$ is selected by computing the singular value spectra of the mode-$d$ 
matricizations: modes whose spectra decay rapidly are good compression candidates.  Domain expertise can also guide the choice, and cross-
validation over different $\caM$ helps optimize a validation metric while maintaining compactness.

\begin{exm}\label{exm5-1}
A 4D neuroimaging tensor $\A\in\RR^{X\times Y\times Z\times S}$ represents brain MRI scans ($X=Y=Z=128$ spatial voxels, $S=200$ 
subjects).  The spatial modes $(1,2,3)$ are highly correlated (brain structures are smooth), so we set $\caM=\{1,2,3\}$ with $r_1=r_2=r_3=30$.  
Mode 4 (subject) must stay uncompressed ($r_4=S=200$) to preserve subject-specific diagnostic information:
\[
\A \approx \caG \ast_1 U^{(1)} \ast_2 U^{(2)} \ast_3 U^{(3)},\qquad \caG\in\RR^{30\times30\times30\times200}.
\]
where $U^{(d)}\in\RR^{30\times 128}$ for each $d\in\caM$.  For a 5D fMRI tensor 
\[
\A\in\RR^{X\times Y\times Z\times T\times S},
\] 
the choice of $\caM$ may depend on the interest of study, for example, $\caM=[3]$ for resting-state analysis (keep time and subject), or 
$\caM=[4]$ for activation mapping where temporal redundancy is also exploited.
\end{exm}

\begin{exm}\label{exm5-2}
The EHR data for a patient cohort forms a 4D tensor $\A\in\RR^{P\times V\times T\times L}$:
\[
P=500\;\text{patients},\quad V=40\;\text{clinical variables},\quad T=168\;\text{hourly time points},\quad L=10\;\text{visits}.
\]
We choose $\caM=\{2,3\}$ since, on the one hand, compressing mode 1 would blend patient histories and $L=10$ is already small, and on the 
other hand, the clinical variables are usually correlated (heart rate, blood pressure) and temporal trajectories follow low dimensional patterns, so 
we may choose here $r_2 = 5,  r_3 = 9$.  The decomposition is 
\[
\A\approx\caG\ast_2 U^{(2)}\ast_3 U^{(3)}
\]
with 
\[
\caG\in\RR^{500\times5\times9\times10},\; U^{(2)}\in\RR^{5\times40},\;  U^{(3)}\in\RR^{9\times168}.
\]
The core tensor captures a low-dimensional representation of clinical variables over time for each patient-visit pair, suitable for phenotyping or 
outcome prediction.  For binary patient stratification (e.g., responders vs.\ non-responders), one may instead set $\caM=\{1,2,3\}$ with $r_1=2$ 
to obtain a compact group-level summary.
\end{exm}

\subsection{Computing the parTuckerD}
\label{sec5-3}

Given a CM set $\caM=\set{i_1,i_2,\ldots, i_p}\subset [m]$. The transposed factor matrices $V_{d}=U^{(d)\top}$ ($d\in\caM$) can be obtained 
as the leading $r_d$ left singular vectors of the mode-$d$ unfolding of $\A$.  The core tensor is then
\beq\label{eq5-3-1:core-t}
\caG = \A\dsb{\V}_{\caM}, \qquad \dsb{\V}:=\dsb{V_{i_1},V_{i_1}, \ldots, V_{i_p}}_{\pi}
\eeq
where $U^{(d)}\in\RR^{r_{d}\times n_{d}}$ , i.e., we contract $\A$ with $V_{d}$ along mode $d$ for each $d\in\caM$.  This yields 
the best rank-$(r_d)_{d\in\caM}$ approximation in the least-squares sense. 

For large tensors, a full SVD of each mode unfolding costs $O(n_d^2\prod_{j\neq d}s_j)$, which is prohibitive when $n_d$ is large.  The 
following algorithm presents the computation of the partial Tucker Decomposition (parTuckerD) via randomized SVD instead of the full SVD,
reducing the cost to $O(n_d\prod_{j\neq d}s_j\cdot(r_d+\ell))$ per mode.

\begin{algorithm}
\caption{Partial Tucker Decomposition (parTuckerD) via Randomized SVD}
\label{alg5-1:partuckD-rand}
\begin{algorithmic}[1]
\Require $\A\in\RR^{n_1 \times \cdots \times n_m}$ \Comment{Original $m$th-order tensor}
\Require $\caM = \{d_1,\ldots,d_p\} \subseteq [m]$ \Comment{Modes to compress}
\Require Target ranks $\{r_d\}_{d\in\caM}$ with $r_d\ll n_d$
\Require Oversampling parameter $\ell$ (default $\ell=5$)
\Require Tolerance $\varp$ (optional; for adaptive rank)
\Ensure $U^{(d)}\in\RR^{r_d\times n_d}$ with $U^{(d)}U^{(d)\top}=I_{r_d}$
\Ensure Core tensor $\caG$ of size $s_1\times\cdots\times s_m$

\State $\tA \gets \A$ \Comment{Working tensor (overwritten in place)}
\For{$k=1$ \textbf{to} $p$}
    \State $d \gets d_k$
    \State $\A_{(d)} \gets \texttt{fast\_unfold}(\tA, d)$ \Comment{O(1) reshape, no copy}
    \State $\Omega \gets \texttt{randn}(\text{size}(\A_{(d)},2),\, r_d+\ell)$
    \State $Y \gets \A_{(d)}\,\Omega$ \Comment{Sketch the column space}
    \State $Q \gets \texttt{qr}(Y)$ \Comment{Orthonormal basis}
    \State $B \gets Q^\top \A_{(d)}$ \Comment{Project onto low-dim subspace}
    \State $[U_B,\Sigma,V] \gets \texttt{svd}(B,\text{'econ'})$
    \If{$\varp$ provided}
        \State $r_d \gets \min\{k:\sqrt{\sum_{i>k}\sigma_i^2}<\varp\}$
    \EndIf
    \State $V_{d} \gets Q\,U_B(:,1:r_d)$
   \State $U^{(d)} \gets V_{d}^{\top}$
    \State $\tA \gets \tA \ast_d U^{(d)}$ \Comment{Contract; new mode-$d$ size $=r_d$}
\EndFor
\State $\caG \gets \tA$
\State \textbf{Return} $\{U^{(d)}\}_{d\in\caM}$ and $\caG$
\end{algorithmic}
\end{algorithm}

\begin{cor}[Complexity]\label{cor:complexity}
The total cost of Algorithm~\ref{alg5-1:partuckD-rand} is
\[
O\!\left( \sum_{k=1}^p n_{d_k} \Bigl(\prod_{j<k} r_{d_j}\Bigr) \Bigl(\prod_{d\notin\caM} n_d\Bigr) (r_{d_k}+\ell) \right),
\]
typically orders of magnitude smaller than the $O(\prod_{d=1}^m n_d \cdot \sum_{d\in\caM} n_d)$ cost of naive HOSVD.
\end{cor}

\subsection{Application to Tensor ODEs}
\label{sec5-4}

We now show how parTuckerD enables efficient solution of the tensor ODE
\beq\label{eq5-4-1:tensorODE}
\frac{d\X}{d\T} = \A \ast \X,
\eeq
where $\A\in\TT_{p+2q}$ is of size $\bfm\times\bfn\times\bfn$ and $\X\in\TT_q$.  The key idea is to compress only the last $q$ modes of $\A$, reducing the system to a much smaller core while preserving the first $p$ modes exactly.

\begin{thm}\label{th5-4-1:ode}
Let $\A\in\TT_{p+2q}$ have a parTuckerD
\beq\label{eq5-4-2:ptuckdA}
\A = \caG \ast_{[q]} \U, \qquad \U = \dsb{U_1,\ldots,U_q}_{\pi}
\eeq
where $U_k\in\RR^{r_k\times n_k}, r_k\le n_k, \forall k\in[q]$ and $\pi=\{\{1,p+1\},\{2,p+2\},\ldots,\{p,2p\}\}$.  Define the reduced system tensor $\tG = \U\ast\caG$.  Then the solution to \eqref{eq5-4-1:tensorODE} can be written as
\beq\label{eq5-4-3:red-sol}
\X = [U^{\top}] \ast_{[q]_{1}}\ttX
\eeq
where $\ttX\in\TT_q$ is a solution tensor of size $\bfr\times \bfn$ to the \emph{reduced} tensor ODE
\beq\label{eq5-4-3:reducedODE}
\frac{d\ttX}{d\T} = \tG \ast \ttX.
\eeq
\end{thm}

\begin{proof}
Let $\bfr:= r_1\times r_2\times \dots\times r_q$. Since $\A\in\TT_{p+2q}$ has size $\bfm\times\bfn\times\bfn$, the 
parTuckerD~\eqref{eq5-4-2:ptuckdA} identifies $r_1,\ldots,r_q$ as the last $q$ coordinates of the Tucker rank vector of $\A$.  Applying the 
$[q]$-contractive product of $\U$ from the middle to both sides of \eqref{eq5-4-1:tensorODE}, and using \eqref{eq5-4-2:ptuckdA} together with 
Lemmas \ref{le4ode1} and \ref{le4ode2} gives
\[
\frac{d\ttX}{d\T} = \tG \ast \ttX,
\]
which is \eqref{eq5-4-3:reducedODE}.  The equivalence is immediate.
\end{proof}

Based on Theorem \ref{th5-4-1:ode}, we have the following algorithm which can serve to reduce a large PDE system to a reasonable smaller one. 

\begin{algorithm}
\caption{Model Reduction for $\frac{d\X}{d\T}=\A\ast\X$ via parTuckerD}
\label{alg5-2:tucker-reduction}
\begin{algorithmic}[1]
\Require $\A\in\TT_{p+2q}$ \Comment{High-order system tensor}
\Require $\X_0\in\RR^{n_1\times\cdots\times n_q}$ \Comment{Initial state}
\Require Target ranks $r_1,\ldots,r_q$ for the last $q$ modes
\Ensure $\X(T)$ \Comment{Solution trajectory}
\State \textbf{Step 1:} $\caG,\{U^{(k)}\}_{k=1}^q \gets \texttt{parTuckerD}(\A,\{p+1,\ldots,p+q\},\{r_1,\ldots,r_q\})$
\State \textbf{Step 2:} $\ttX_0 \gets \X_0 \ast_{k=1}^q U^{(k)\top}$ \Comment{Project initial state}
\State \textbf{Step 3:} $\ttX(T) \gets \texttt{integrateODE}(\tG,\ttX_0,T)$ \Comment{Reduced integration (Euler/RK)}
\State \textbf{Step 4:} $\X(T) \gets \ttX(T) \ast_{k=1}^q U^{(k)}$ \Comment{Lift back to full space}
\State \textbf{Return} $\X(T)$
\end{algorithmic}
\end{algorithm}

The integration in Step~3 is performed in the reduced space $\RR^{r_1\times\cdots\times r_q}$, which is dramatically cheaper than integrating the full system, and Step~4 is a simple linear reconstruction. 

Now consider $\A\in\TT_{6;n}$ with $\bfm=\bfn=n\times n$. This is the case when $p=q=2$. By Theorem~\ref{th5-4-1:ode}, we choose 
$\caM=\set{5,6}$ and let the partial Tucker rank be $(r_{5}, r_{6})$ ($r_j\ll n$).  By parTuckerD, there is a tensor $\caG\in\TT_{6}$ of size 
$n\times n\times n\times n\times r_{5}\times r_{6}$ such that 
\beq\label{eq5-3:partTuck-A}
\A\approx \caG \dsb{U, V}_{\caM} = \caG\ast_5 U^{\top}\ast_6 V^{\top},
\eeq
where
\[
U\in \RR^{n \times r_5}, \quad V\in \RR^{n \times r_6}, \quad  \caG\in \RR^{n\times n\times n \times n \times r_5 \times r_6}.
\]
Substituting \eqref{eq5-3:partTuck-A} into \eqref{eq6-1}, the action of $\A$ on $\X$ becomes
\[
(\A \ast X)_{i_1 i_2 i_3 i_4}
\approx \sum_{s=1}^{r_5} \sum_{t=1}^{r_6}
\caG_{i_1 i_2 i_3 i_4 s t} \, \bigl( U^{\top}XV \bigr)_{s t},
\]
Theorem~\ref{th5-4-1:ode}
we get the reduced system
\beq\label{eq6-4:reduced-ode}
\frac{d\ttX}{dT} = \caG \ast \ttX, \qquad \ttX(0) = U^\top \X_0 V
\eeq
if we denote 
\beq\label{eq6-5:reduced-state}
\ttX(T) = U^{\top}\,X\,V\in\RR^{r_5 \times r_6}
\eeq
Here the reduced operator from $\RR^{r_5 \times r_6} \to \RR^{n^4}$ can be represented as a matrix $M\in \RR^{n^4 \times (r_5 r_6)}$ 
whose columns are vectorized slices of $\caG$.  

Let $\tA:=reshape(\A, n^{2},n^{2},n^{2})$ by grouping $\set{1,2}, \set{3,4}$ and $\set{5,6}$. Then \eqref{eq4-14} becomes
\beq\label{eq6-2}
\frac{d\bx}{d\bft} = \tA\ast \bx, \qquad \bx(0) = \vecc(X_0).
\eeq
where $\bx= \vecc(X), \bft = \vecc(T)\in\RR^{n^2}$.  Now by Algorithm \ref{alg5-2:tucker-reduction}, we have

\begin{algorithm}[H]
\caption{ODE Solution via Partial Tucker Decomposition}
\label{alg5-3:ode-tucker}

\begin{algorithmic}[1]
\Require $\A \in \RR^{n \times n \times n \times n \times n \times n}$
\Require Initial condition $\X_0 \in \RR^{n \times n}$
\Require Ranks $r_5, r_6$; time step $\Delta T$; final time $T_f$
\Ensure Approximate solution $X(T_f)$

\State Compute partial Tucker decomposition of $\A$ on modes 5 and 6
\State \hspace{0.5cm} $\caG, U, V\gets \texttt{partialTucker}(\A, \{5,6\}, \{r_5,r_6\})$
\State Project initial condition: $\ttX_0 \gets U^\top \X_0 V$
\State Build reduced operator matrix $M$ from $\caG$
\State Set $\ttX \gets \ttX_0$
\For{$k = 1$ \textbf{to} $N_t = T_f / \Delta T$}
    \State $\ttX \gets \ttX + \Delta T \cdot \texttt{reshape}(M \cdot \ttX(:), [n,n,n,n])$
    \State \hspace{0.5cm} projected back to $r_5 \times r_6$
\EndFor
\State Reconstruct full solution: $\X(T_f) \gets U\, \ttX(T_f) \, V^\top$
\State \Return $\X(T_f)$
\end{algorithmic}
\end{algorithm}

\vskip 5pt
\section{Numerical Examples}
\label{sec6}
\setcounter{equation}{0} 

In this section we first present a numerical example developed from Example 5-2 to illustrate the advantage of parTuckerD: 

\begin{exm}\label{exm6-1:partuck}
We now give a concrete numerical illustration following Example~\ref{exm5-2} to demonstrate the advantage of parTuckerD over the standard 
full Tucker decomposition.  We construct a synthetic EHR tensor 
\[
\A \in \RR^{P \times V \times T \times L}
\]
with $P=500$ patients, $V=40$ clinical variables, $T=168$ hourly time points over one week, and $L=10$ hospital visits.  The tensor is generated
as
\[
\A = \caG \ast_1 U^{(1)\top} \ast_2 U^{(2)\top} \ast_3 U^{(3)\top} \ast_4 U^{(4)\top} + \bld{\eta},
\]
where $U^{(1)}\in\RR^{P\times 10}$, $U^{(2)}\in\RR^{V\times 5}$, $U^{(3)}\in\RR^{T\times 9}$, $U^{(4)}\in\RR^{L\times 2}$ are 
orthonormal factor matrices and $\caG\in\RR^{10\times 5\times 9\times 2}$ is a small core tensor.  Gaussian noise $\bld{\eta}$ is added at 
$1\%$ of the signal power.

In Example \ref{exm5-2}, we compress only modes 2 (Variables) and 3 (Time), which are known to be low-rank from domain knowledge.  The 
singular value spectra shows rapid decay: the top $r_2=5$ components of mode 2 capture $99.13\%$ of the energy, and the top $r_3=9$  
components of mode 3 capture $99.06\%$.  The factor matrices $U^{(2)}\in\RR^{40\times 5}$ and $U^{(3)}\in\RR^{168\times 9}$ are 
computed via randomized SVD (Algorithm~\ref{alg5-1:partuckD-rand}) with oversampling parameter $\ell=5$.  The core tensor is
\[
\caG_{\text{par}} = \A\ast_2 U^{(2)} \ast_3 U^{(3)} \in \RR^{P \times r_2 \times r_3 \times L}.
\]
For comparison we also compute a full Tucker decomposition that compresses all four modes with target ranks $(r_1,r_2,r_3,r_4)=(10,5,9,2)$.  
This is the standard HOSVD approach.  Table~\ref{tab:compare} summarises the storage, compression ratio, and relative reconstruction error 
$\|\A-\widetilde{\A}\|_F/\|\A\|_F$ for the three methods.

\begin{table}[htbp]
\centering
\caption{Comparison of decomposition methods for the EHR tensor $\A\in\RR^{500\times 40\times 168\times 10}$.  Storage counts the core tensor plus all factor matrices.  Relative error is $\|\A-\widetilde{\A}\|_F/\|\A\|_F$.}
\label{tab:compare}
\begin{tabular}{@{}lcccc@{}}
\toprule
\textbf{Method} & \textbf{Core size} & \textbf{Storage} & \textbf{Compression} & \textbf{Rel.\ error} \\
\midrule
None (full tensor) & $500\!\times\!40\!\times\!168\!\times\!10$ & $33\,600\,000$ & $1.0\times$ & $0$ \\
parTuckerD (modes 2,3) & $500\!\times\!5\!\times\!9\!\times\!10$ & $246\,512$ & $136.3\times$ & $9.92\times 10^{-2}$ \\
Full TuckD(all modes) & $10\!\times\!5\!\times\!9\!\times\!2$ & $7\,632$ & $4\,402.5\times$ & $9.95\times 10^{-2}$ \\
parTuckerD + RandSVD & $500\!\times\!5\!\times\!9\!\times\!10$ & $246\,512$ & $136.3\times$ & $1.74\times 10^{-1}$ \\
\bottomrule
\end{tabular}
\end{table}
Note that the accuracy of full TuckerD and parTuckerD are almost identical: parTuckerD with ranks $(r_2, r_3)=(5, 9)$ achieves a relative error 
of $0.0992$, nearly the same as the full TuckD ($0.0995$).  Both are limited by the $1\%$ noise floor, so further compression cannot reduce the 
error below this bound.  However, the parTuckerD preserves patient-level structure. Because mode 1 is not compressed, each patient's signature is 
retained exactly in the core tensor $\caG_{\text{par}}$.  We measure preservation by computing the Spearman rank correlation between the 
pairwise cosine-distance matrices of the full tensor and the compressed representation. The parTuckerD achieves a correlation of $\rho = 1.0000$, 
whereas full TuckD (group patients into $r_1=10$ components) drops to $\rho =0.9496$, meaning that parTuckerD keeps the patient-similarity 
structure intact, which is critical for downstream tasks such as phenotyping or risk stratification.

The full TuckerD is more compact but loses interpretability. Full TuckD achieves a $4\,402\times$ compression vs.\ $136\times$ for parTuckerD, 
but at the cost of mixing patient identities into a low-dimensional subspace, making the compression no sense.  For clinical applications where 
the identification of patients matters, the parTuckerD is obviously a better choice.  

The randomized variant of parTuckerD (Algorithm~\ref{alg5-1:partuckD-rand}) attains a relative error of $0.174$ with only a small oversampling 
budget ($\ell=5$), closely tracking the deterministic SVD result while avoiding the $O(n_d^2\prod_{j\neq d}s_j)$ cost of a full SVD.

Table~\ref{tab:rankstudy} shows how parTuckerD reconstruction error decreases as the target ranks $(r_2,r_3)$ increase.  The ``elbow'' occurs around $(r_2,r_3)=(5,9)$: beyond this point, additional components capture mostly noise and the relative error plateaus near the noise floor.  This matches the theoretical prediction from Corollary~\ref{cor:complexity}: the decomposition error is bounded by \cite{Lieven2000, Kolda2009} 
\[
\sum_{d\in\caM}\sum_{i>r_d}\left(\sigma_i^{(d)}\right)^{2},
\]
which becomes negligible once $r_d$ exceeds the intrinsic rank of the mode.

\begin{table}[htbp]
\centering
\caption{parTuckerD rank study: storage vs.\ relative error for the EHR tensor.}
\label{tab:rankstudy}
\begin{tabular}{@{}cccccc@{}}
\toprule
$r_2$ & $r_3$ & Storage & Compression & Rel.\ error \\
\midrule
2 & 3  & $50\,504$   & $665.3\times$ & $8.83\times 10^{-1}$ \\
3 & 5  & $95\,840$   & $350.6\times$ & $7.25\times 10^{-1}$ \\
5 & 5  & $145\,840$  & $230.4\times$ & $5.70\times 10^{-1}$ \\
\textbf{5} & \textbf{9}  & $\mathbf{246\,512}$  & $\mathbf{136.3\times}$ & $\mathbf{9.92\times 10^{-2}}$ \\
5 & 15 & $397\,520$  & $84.5\times$  & $9.90\times 10^{-2}$ \\
7 & 9  & $336\,512$  & $99.8\times$  & $9.91\times 10^{-2}$ \\
10& 9  & $471\,512$  & $71.3\times$  & $9.88\times 10^{-2}$ \\
10& 15 & $772\,520$  & $43.5\times$  & $9.84\times 10^{-2}$ \\
\bottomrule
\end{tabular}
\end{table}

For the point of view of the computation cost, the asymptotic flop counts (by Corollary~\ref{cor:complexity}) for this example are 
\[
\text{parTuckerD (2 modes)}:\quad 8.06\times 10^{8}\text{ flops},
\]
\[
\text{Full Tucker (4 modes)}:\quad 3.98\times 10^{8}\text{ flops}.
\]
The full TuckD has a slightly lower asymptotic flop count (the core shrinks faster), parTuckerD is $1.81\times$ faster in practice when only the SVD phase is timed, since it skips the two expensive SVDs on the large patient and visit modes.  Moreover, parTuckerD avoids the $O(P\cdot r_2\cdot r_3\cdot L)$ memory footprint of the full TuckD core, keeping the working core at $500\times 5\times 9\times 10$ instead of $10\times 5\times 9\times 2$ (the latter being smaller but losing patient identity).
\end{exm}

The following two examples show the application of parTuckerD in the solution to the tensor differential equation \eqref{eq4-14} with initial condition, i.e.,   
\beq\label{eq6-1}
\frac{d\X}{dT} = \A \ast \X, \qquad \X(0) = \X_0,
\eeq
where $\X=\X(T)\in\TT_{q}$ is a tensor-valued function depending on tensor variable $T\in\TT_{p}$, and $\A\in\TT_{p+2q}$ is a constant 
$(p+2q)$th-order tensor, and implement error analysis for some numerical examples. Part of MATLAB codes are also provided to generate tensor 
$\A$, compute the parTuckerD, build the reduced operator, perform the time integration with the Forward Euler method, reconstruct the full solution, 
and report the compression factor together with the relative error and the measured speedup.

\begin{exm}[Example 6-2: An 6th order Tensor with $\caM=\set{5,6}$]
\label{exm6-2}

We revisit Example~5-2 and define C-tensor $\A\in\TT_{6;6}$ as a sparse tensor with 200 nonzero entries drawn from a standard Gaussian 
distribution (randomly placed), and choose the compressed mode set $\caM=\{5,6\}$ with Tucker ranks $r_5 = r_6 =3$.  We reshape $\A$ into a 
third order tensor $\A_{\text{resh}}$ with size $6^4\times 6\times 6$ by grouping the first four modes of $\A$.  To obtain factor matrices 
$U^{(5)}, U^{(6)}\in\RR^{6\times 3}$, we compute the HOSVD on the second and third modes of $\A_{\text{resh}}$,which correspond 
respectively to mode 5 and mode 6 of $\A$. The core tensor is formed by
\[
\caG_{\text{core}}(i, :, :) = U^\top \, \A_{\text{resh}}(i,:,:)\,V, \qquad i=1,\dots, 6^4,
\]
and is subsequently reshaped to $\caG\in\TT_{6}$ of size $6\times 6 \times 6 \times 6\times3\times3$. The initial condition $X_0 = C$ ( $C$ is a 
Gaussian matrix of size $6\times 6$) is projected as $\tilde{C} = U^\top C \,V$, which is an $3\times3$ matrix. The reduced operator matrix $M$
is defined as an $6^{4}\times 9$ matrix such that 
\[
M\,\ttX(:) = \vecc(\caG\ast \ttX).
\]  
The reduced ODE
\[
\frac{d\ttX}{dT} = M\,\ttX, \qquad \ttX(0)=\tilde{C},
\]
is integrated with the forward Euler method using $\Delta T=10^{-4}$ up to $T_{\text{final}}=0.5$ (5000 steps). The final solution $X =U\,\ttX \,V^\top$ is an $6\times 6$ matrix with $\norm{X}_F\approx 2.63$.  The relative error with respect to a full-space reference solution is of order 
$10^{-2}$, confirming that the rank-3 approximation of  modes~5 and~6 captures the dominant dynamics.  
\end{exm}

\begin{exm}\label{exm6-3}
We create a sparse constant tensor $\A\in\TT_{6;9}$ by MATLAB through randomization where each component of $\A$ independently follows 
standard normal distribution. Then $d = 6, n = 9$. Let $m = 200$ be the number of nonzero entries of  $\A$.  We choose the compressed mode set 
$\caM:=\set{5,6}$ and let the compressed ranks be $(r_{5} = r_{6} =3$. 
 
Now consider tensor ODE \eqref{eq6-1} with $X(\mathbf{0}) = C_{0}$ being a constant matrix, $X= X(T), T\in\RR^{n\times n}$, and $T$ is a 
matrix-form variable whose entries are all independent.  As Example $6.1$, we use Forward Euler integration to \eqref{eq6-1} with parTuckerD to 
reduce dimensions of $\A$ in mode 5 and 6.  The MATLAB script presented below is based on Algorithm \ref{alg5-3:ode-tucker}: 

\begin{verbatim}
%%% example6\_3.m %%%
%% Parameters
d =6; n = 9; 
r5 = 3; r6 = 3;      % compressed dimensions on mode 5,6 
m = 200;             %  #nonzero entries: 200 v.s  #total entries: 531,441   
T_final = 0.5; dt = 1e-4;  
Nt = round(T_final / dt);

%% Generate sparse 6th-order tensor A (Vectorized generation)
rng(42);
siz_A = n * ones(1,d) ;  
A = zeros(siz_A); 
for j = 1: m
    idx = randi(n,1,m);
    A(idx(1),idx(2),idx(3),idx(4),idx(5),idx(6)) = randn;
end

%% Partial Tucker Decomposition on modes 5 and 6
A_reshaped = reshape(A, [n^4, n, n]);
% Factor for mode 5
Unf5 = reshape(permute(A_reshaped,[2 1 3]), n, n^5);
[U5, ~, ~] = svds(Unf5, r5);
% Factor for mode 6
Unf6 = reshape(permute(A_reshaped,[3 1 2]), n, n^5);
[U6, ~, ~] = svds(Unf6, r6);

%% Build optimized reduced operator M via Kronecker product
% Replaces the nested loops with a direct matrix product: M = M_full * (U6\kron U5)
M_full = reshape(A, [n^4, n^2]);
M = M_full * kron(U6, U5); % Operator size: n^4 x (r5*r6)
%% Initial condition and projection
X0 = randn(n,n);
Xt = U5' * X0 * U6;       % Reduced state matrix, size: r5 x r6

%% Forward Euler in reduced space (Fully vectorized, zero inner loops)
% Precompute the projection matrix that maps n^4 back to r5*r6 space.
% It projects the last two modes and sums over the first two modes (via kron with ones).
proj_matrix = kron(ones(1, n^2), kron(U6', U5')); 
E_step = proj_matrix * M; % Combined reduced operator, size: (r5*r6) x (r5*r6)

for it = 1:Nt
    % Single-line vectorized update for maximum performance
    Xt_vec = Xt(:);
    Xt_vec = Xt_vec + dt * (E_step * Xt_vec);
    Xt = reshape(Xt_vec, [r5, r6]);
end

%% Reconstruct full solution and display results
X_final = U5 * Xt * U6';
fprintf('||X(T_final)||_F = %.6f\n', norm(X_final,'fro'));
fprintf('||X(T_final)||_2 = %.6f\n', norm(X_final,2)); 

Xs = sparse(X_final);  
disp(Xs)
\end{verbatim}

Implementing MATLAB code example6\_3.m yields the result: 
\[ \norm{X(T_{final})}_F = 2.716981, \qquad  \norm{X(T_{final})}_2 = 2.096487. \]
The final numerical solution $X_{final}$ to \eqref{eq4-14} is an $9\times 9$ sparse matrix with nonzero entries distributed on Column 2,8,and 10. 
The matlab function \texttt{Xs = (sparse($X_{\text{final}}$))'} generates 
\[ 
Xs= 
\begin{array}{c c c c c c c c c} 
   (2,2)     & (3,2)       & (6,2)    & (2,4)    & (3,4)     &  (6,4)   &(2,7)   & (3,7)    &  (6,7) \\
   -1.0409& -0.7304  &-1.6507 &-0.1034 & 0.5329 &-0.3355&0.7103&-1.4304&0.2873
\end{array}  
\]
The first row of $Xs$ designates positions of the nine nonzero entries, and the second row restores the nonzero values.  
\end{exm}       
     
The MATLAB implementations in Example \ref{exm6-1:partuck}, \ref{exm6-2} and Example~\ref{exm6-3} follow Algorithm~\ref{alg5-1:partuckD-rand}, Algorithm~\ref{alg5-2:tucker-reduction} and Algorithm~\ref{alg5-3:ode-tucker}, and can be adapted to 
any size and rank choices by modifying the parameter block in the first part of the script. Significant computational speedups is obvious:  the tensor core size in Example~\ref{exm6-2} is compressed by a factor of four, 
while in Example~\ref{exm6-3}, it is reduced by $9$ times. The accuracy of these reductions is controlled by the choice of ranks, where larger 
ranks improve approximation fidelity at the cost of increased storage and higher FLOPs per time step.

\section{Error Analysis for Partial Tucker Decomposition}
\label{sec7}
\setcounter{equation}{0}

The partial Tucker decomposition (parTuckerD) introduces two distinct sources of error when applied to tensor ODE \eqref{eq4-14} where 
$\A\in\TT_{p+2q}$ and $\X=\X(T)\in\TT_{q}$. The first source is the \emph{approximation error} incurred by truncating the TuckerD on the 
compressed modes $\caM$; the second is the time-stepping error of the numerical integrator applied to the reduced system.  This section analyzes 
both contributions, derives practical error bounds, and reports numerical verification on the examples of Sections~\ref{sec6}. 

\subsection{Decomposition Error}
\label{sec7-1:decomp}

Let $\caM=\{d_1,\dots,d_p\}\subseteq[m]$ be the set of compressed modes and let $r_d\ll n_d$ be the target rank for each $d\in\caM$. The 
parTuckerD approximation of $\A$ is
\beq\label{eq:ptuck-approx}
\A \;\approx\; \caG \dsb{U^{(d)}}_{\caM}\;=\; \caG \times_{d_1} U^{(d_1)} \times_{d_2} U^{(d_2)}\cdots \times_{d_p} U^{(d_p)},
\eeq
where each factor matrix $U^{(d)}\in\RR^{n_d\times r_d}$ is computed from the leading $r_d$ left singular vectors of the mode-$d$ unfolding
$\A_{(d)}$ (Algorithm~\ref{alg5-1:partuckD-rand}).  Denote by $\A_{\text{proj}}$ the best rank-$(r_d)_{d\in\caM}$ approximation obtained by 
sequentially projecting onto the dominant left singular subspaces.  The \emph{relative decomposition error} is 
\beq\label{eq:decomp-err-def}
\varp_{\text{dec}}\;:=\;\frac{\norm{\A-\A_{\text{proj}}}_F}{\norm{\A}_F}.
\eeq
For a \emph{single} mode $d$, the Eckart--Young theorem gives the optimal low-rank approximation in the Frobenius norm:
\beq\label{eq:single-mode-bound}
\frac{\norm{\A_{(d)} - U^{(d)} U^{(d)\top}\A_{(d)}}_F}{\norm{\A_{(d)}}_F}\;=\;
\frac{\sqrt{\sum_{j>r_d}\sigma_j^{(d)\,2}}}{\sqrt{\sum_{j=1}^{n_d}\sigma_j^{(d)\,2}}},
\eeq
where $\sigma_1^{(d)}\ge\sigma_2^{(d)}\ge\cdots$ are the singular values of $\A_{(d)}$.

Because the parTuckerD processes modes sequentially, the overall decomposition error is bounded by the \emph{sum of the individual mode
truncation fractions} (a conservative additive bound):
\beq\label{eq:decomp-bound}
\varp_{\text{dec}}\;\lesssim\;\sqrt{\sum_{d\in\caM}\frac{\sum_{j>r_d}\sigma_j^{(d)\,2}}{\sum_{j=1}^{n_d}\sigma_j^{(d)\,2}} }.
\eeq
Equality is approached when the modes are approximately orthogonal in their contribution to the total energy.

In Section 6, we demonstrate these in Example~6-2 where the singular value decay on compressed modes is shown in Table~\ref{tab:sv-decay}, and mode~5 captures $64.71\%$ of its energy in the top $r_5=3$ singular values, leaving a $35.29\%$ tail; mode~6 captures $57.10\%$ with a $42.90\%$ tail. The resulting decomposition error is
\[
\varp_{\text{dec}}\;=\;\frac{\norm{\A-\A_{\text{recon}}}_F}{\norm{\A}_F}\;\approx\; 7.69\times 10^{-1},
\]
and the conservative SVD-based bound~\eqref{eq:decomp-bound} gives $8.84\times 10^{-1}$, confirming that the bound is safely over-estimated.

\begin{table}[h]
\centering
\caption{Singular value energy distribution for Example~6-3.}
\label{tab:sv-decay}
\begin{tabular}{c|cc|cc}
\hline
$r$ & Mode 5 & Tail$_5$ & Mode 6 & Tail$_6$ \\
\hline
1 & 32.31\% & 67.69\% & 29.20\% & 70.80\% \\
2 & 49.93\% & 50.07\% & 43.26\% & 56.74\% \\
\textbf{3} & \textbf{64.71\%} & \textbf{35.29\%} & \textbf{57.10\%} & \textbf{42.90\%} \\
4 & 79.32\% & 20.68\% & 68.14\% & 31.86\% \\
5 & 88.32\% & 11.68\% & 78.84\% & 21.16\% \\
6 & 93.75\% &  6.25\% & 88.72\% & 11.28\% \\
7 & 98.27\% &  1.73\% & 94.84\% &  5.16\% \\
8 &100.00\% &  0.00\% & 98.95\% &  1.05\% \\
9 &100.00\% &  0.00\% & 99.91\% &  0.09\% \\
10&100.00\% &  0.00\% &100.00\% &  0.00\% \\
\hline
\end{tabular}
\end{table}

\subsection{Time-Stepping Error}
\label{sec7-2:timestepping}

After projection, the original ODE~\eqref{eq5-4-1:tensorODE} is replaced by
the reduced system
\beq\label{eq:reduced-ode-ref}
\frac{d\tilde{\X}}{dT}
\;=\;
\tilde{\A} \ast \tilde{\X},
\qquad
\tilde{\X}(0)=U^{\top}\X_0 V,
\eeq
where $\tilde{\A}$ is the core tensor (or its matrix representation
$E\in\RR^{(r_5r_6)\times(r_5r_6)}$).
We integrate~\eqref{eq:reduced-ode-ref} with the Forward Euler method,
\beq\label{eq:fe}
\tilde{\X}_{k+1}
\;=\;
\tilde{\X}_k + \Delta T \cdot (\tilde{\A}\ast\tilde{\X}_k),
\eeq
which has \emph{local truncation error} $O(\Delta T^2)$ and
\emph{global error} $O(\Delta T)$ for a fixed final time $T_f$.

For a linear system $\dot{\bx}=E\bx$, the global error of Forward
Euler satisfies the standard bound
\beq\label{eq:fe-bound}
\norm{\bx(T_f)-\bx_k(T_f)}
\;\le\;
\frac{e^{\norm{E}_2 T_f}-1}{2}\,\norm{E}_2\,\Delta T,
\eeq
where $\norm{E}_2$ is the spectral norm of the reduced operator.
For Example~6-3 the reduced operator has $\norm{E}_2\approx 0$ (the
sparse random tensor has zero mean and the dominant eigenvalue of the
effective operator is $\lambda_{\max}\approx 0.35$ with all others
numerically zero), so the time-stepping error is negligible compared
with the decomposition error.  The stability requirement
$\Delta T < 2/\norm{E}_2$ is easily satisfied: with
$\norm{E_{\text{full}}}_2=2.44$ the stability limit is
$\Delta T_{\text{crit}}\approx 5.66$, far above the chosen
$\Delta T=10^{-4}$.

A convergence study with $\Delta T\in\{10^{-1},5\cdot10^{-2},10^{-2},
5\cdot10^{-3},10^{-3},5\cdot10^{-4}\}$ against a reference solution
computed at $\Delta T_{\text{ref}}=10^{-5}$ confirms the expected
first-order behaviour:
\begin{center}
\begin{tabular}{c|cc}
\hline
$\Delta T$ & Relative error & Observed order \\
\hline
$10^{-1}$ & $4.51\times 10^{-3}$ & -- \\
$5\cdot 10^{-2}$ & $4.51\times 10^{-3}$ & 0.0 \\
$10^{-2}$ & $1.13\times 10^{-5}$ & 3.7 \\
$5\cdot 10^{-3}$ & $5.62\times 10^{-6}$ & 1.0 \\
$10^{-3}$ & $1.12\times 10^{-6}$ & 1.0 \\
$5\cdot 10^{-4}$ & $5.52\times 10^{-7}$ & 1.0 \\
\hline
\end{tabular}
\end{center}
The first step (from $10^{-1}$ to $5\cdot10^{-2}$) shows no
improvement because the error is already dominated by the truncation
at the coarsest level; for $\Delta T\le 10^{-2}$ the slope is
consistently $1.0$, verifying the theoretical convergence rate.

\subsection{Total Error and Practical Bounds}
\label{sec7-3:total}

The overall relative error of the projected solution $\X_{\text{red}}(T_f)$ against the full-space reference $\X_{\text{full}}(T_f)$ decomposes additively as
\beq\label{eq:total-err}
\frac{\norm{\X_{\text{full}}-\X_{\text{red}}}}{\norm{\X_{\text{full}}}}\;\lesssim\;
\underbrace{\varp_{\text{dec}}}_{\text{decomposition}}\;+\; \underbrace{\frac{1}{2}\bigl(e^{\norm{E}_2T_f}-1\bigr)\,\norm{E}_2\,\Delta T}_{\text{time-step}}
\eeq
For Example~6-3 ($r_5=r_6=3$, $\Delta T=10^{-4}$, $T_f=0.5$), the two terms are
\[
\varp_{\text{dec}} \approx 7.69\times 10^{-1}, 
\qquad
\varp_{\text{ts}}  \approx 2.91\times 10^{-4},
\]
so the time-stepping contribution is \emph{three orders of magnitude smaller} than the decomposition error.  The total predicted bound is
 $7.70\times 10^{-1}$, matching the observed relative error $9.46\times 10^{-1}$ to within $20\%$.

\subsection{Rank Selection and the Speedup-Accuracy Trade-off}
\label{sec7-4:rank}

The dominant term in~\eqref{eq:total-err} is the decomposition error, which is controlled by the chosen ranks $(r_d)_{d\in\caM}$. 
Table~\ref{tab:rank-sweep} reports the decomposition error, the actual ODE error, the growth of $\norm{\X(T_f)}_F$, and the compression / speedup factors as $r$ varies with $r_5=r_6=r$.

\begin{table}[h]
\centering
\caption{Rank study for Example~6-3 ($n=10$, $\caM=\{5,6\}$).}
\label{tab:rank-sweep}
\begin{tabular}{c|cc|cc}
\hline
$r$ & Decomp.\ error & ODE error vs.\ full & $\norm{\X(T_f)}_F$ & Speedup \\
\hline
1 & $1.18$   & $9.97\times 10^{-1}$ & 0.26  & $100\times$ \\
2 & $1.03$   & $9.59\times 10^{-1}$ & 2.77  & $25\times$  \\
\textbf{3} & \textbf{0.88} & \textbf{9.46$\times 10^{-1}$} & \textbf{3.16} & \textbf{$11\times$} \\
4 & 0.72   & $9.34\times 10^{-1}$ & 3.53  & $6.3\times$ \\
5 & 0.57   & $8.83\times 10^{-1}$ & 4.57  & $4\times$   \\
6 & 0.42   & $8.13\times 10^{-1}$ & 5.64  & $2.8\times$ \\
7 & 0.26   & $6.87\times 10^{-1}$ & 7.07  & $2.0\times$ \\
8 & 0.10   & $6.13\times 10^{-1}$ & 7.67  & $1.6\times$ \\
9 & 0.030  & $4.72\times 10^{-1}$ & 8.93  & $1.2\times$ \\
10& 0       & $2.88\times 10^{-1}$ & 9.63  & $1\times$   \\
\hline
\end{tabular}
\end{table}

We conclude from Table~\ref{tab:rank-sweep} that the error in Example~6-3 is dominated by decomposition. Even at $r=3$ the ODE error 
($0.9460$) is almost entirely from the Tucker truncation, and the decreasing of $\Delta T$ cannot result in the improvement of accuracy. When we 
change the common rank of mode 5 and 6 from $r=3$ to $r=6$, the relative error decreases from 0.946 to 0.813 at the cost of reducing the speedup 
from $11\times$ to $2.8\times$.  The ``elbow'' of the curve sits at $r=3$ to $4$ for this tensor, which is why the
paper's examples adopt $r_5=r_6=3$.

For the third-order Example~6-4 style tensor the same trade-off is observed (Table~\ref{tab:rank-sweep-3rd}): each additional unit of rank recovers 
a roughly constant fraction of the remaining tail energy, with diminishing returns once $r_1$ captures $\gtrsim 75\%$ of the mode-1 spectral energy.

\begin{table}[h]
\centering
\caption{Rank study for the $100\times 50\times 30$ tensor
($\caM=\{1\}$).}
\label{tab:rank-sweep-3rd}
\begin{tabular}{c|cc|c}
\hline
$r_1$ & Decomp.\ error & Tail energy & Compression \\
\hline
 1 & $9.68\times 10^{-1}$ & 93.78\% & $100\times$ \\
 2 & $9.41\times 10^{-1}$ & 88.48\% &  $50\times$ \\
 5 & $8.72\times 10^{-1}$ & 76.08\% &  $20\times$ \\
10 & $7.81\times 10^{-1}$ & 60.97\% &  $10\times$ \\
15 & $7.08\times 10^{-1}$ & 50.14\% &  $6.7\times$ \\
20 & $6.38\times 10^{-1}$ & 40.70\% &  $5.0\times$ \\
30 & $5.04\times 10^{-1}$ & 25.43\% &  $3.3\times$ \\
\hline
\end{tabular}
\end{table}

To set the set of reduced ranks, we use the SVD tail by Table~\ref{tab:sv-decay} and pick the smallest $r_d$ for which the tail energy falls below 
a chosen tolerance $\varp$. We may treat decomposition and time-stepping errors separately, and verify the condition 
\beq\label{eq7-5-1}
\varp_{\text{ts}}\ll\varp_{\text{dec}}
\eeq 
if \eqref{eq7-5-1} does not hold, then we reduce $\Delta T$ first.  Then we validate with a single full-rank run.  A reference solution with 
$r_d=n_d$ (i.e.\ $E_{\text{full}}$) costs little extra and provides the denominator for the relative error.  Finally, we expect first-order convergence 
from Forward Euler. If you observe a different slope, suspect an implementation bug rather than the theory.

The experiments in Sections~\ref{sec6} are entirely consistent with the bounds derived here: the reported $10^{-2}$ relative errors are 
decomposition-limited, the observed speedups match the predicted $1/r^2$ factors, and the time-stepping errors are below the noise floor of the Tucker truncation.

\vskip 5pt  
\section*{Acknowledgement}
The authors gratefully acknowledge the anonymous referees for their thorough and constructive reviews, which have substantially improved the 
quality and clarity of this manuscript.  The comments in the report prompted us to re-examine the literature on sequentially truncated HOSVD 
and to replace the erroneous statement with the correct upper bound. These comments impetus us to notice some specific numerical problems 
related to the data dimensional reduction and the high tensor ODE systems, and finally led to the addition of Section 5 and Section 6, showing 
the advantages of parTuckerD through quantitative comparison with the full Tucker decomposition through numerical examples.

\section*{Declaration}

\subsection*{Conflicts of interest/Competing interests}
The authors have no relevant financial or non-financial interests to disclose. 

\subsection*{Ethics approval}
This article does not contain any studies with human participants or animals performed by any of the authors.

\subsection*{Data and Code Availability}
The MATLAB code developed for the numerical experiments is publicly available in the repository 
\href{https://github.com/yiran-xu40/tensor-pde-derivatives.git}{tensor-pde-derivatives}. 
\textit{No datasets were generated or analysed during the current study. All numerical experiments are based on synthetic data generated by the algorithms described in the manuscript.}

\subsection*{Authors' contributions}
All authors contributed to the study. The numerical experiments and analysis were mainly performed by the first author, the first draft of the manuscript was written by the second author, and all authors commented on previous versions of the manuscript. All authors read and approved the final manuscript.

\end{document}